\documentclass[12pt]{article}

\usepackage{graphicx}
\usepackage{amsmath,amsthm,amssymb,color,bbm}
\usepackage[noadjust,sort]{cite}
\usepackage{bbm}

\usepackage{tikz}
\usepackage{enumitem}
\usepackage[capitalise]{cleveref}
\usepackage[font=normalsize]{caption}

\makeatletter
\g@addto@macro\normalsize{%
  \setlength\abovedisplayskip{8pt plus 3pt minus 3pt}
  \setlength\belowdisplayskip{8pt plus 3pt minus 3pt}
  \setlength\abovedisplayshortskip{6pt plus 3pt minus 2pt}
  \setlength\belowdisplayshortskip{6pt plus 3pt minus 2pt}
}
\makeatother

\date{\today}

\normalsize

\numberwithin{equation}{section}

\def\({\bigl(}
\def\){\bigr)}

\newtheorem{thm}{Theorem}[section]
\newtheorem{cor}[thm]{Corollary}
\newtheorem{lemma}[thm]{Lemma}
\newtheorem{conj}[thm]{Conjecture}

\theoremstyle{definition} 
\newtheorem{remark}[thm]{Remark}

\def\abs#1{\lvert#1\rvert}

\def\dfrac#1#2{\lower0.15ex\hbox{\large$\textstyle\frac{#1}{#2}$}}
\def\({\bigl(}
\def\){\bigr)}
\def\st{\mathrel{:}}

\def\leq{\leqslant}
\def\geq{\geqslant}
\let\eps=\varepsilon

\def\U{\boldsymbol{U}}
\def\W{\boldsymbol{W}}
\def\X{\boldsymbol{X}}
\def\Y{\boldsymbol{Y}}
\def\Z{\boldsymbol{Z}}
\def\T{\boldsymbol{T}}

\def\calA{\mathcal{A}}
\def\calF{\mathcal{F}}
\def\calG{\mathcal{G}}

\def\thetavec{\boldsymbol{\theta}}

\def\wvec{\boldsymbol{w}}
\def\xvec{\boldsymbol{x}}
\def\yvec{\boldsymbol{y}}

\def\uvec{\boldsymbol{u}}

\def\mysup{\operatorname{sup}}

\def\E{\operatorname{\mathbb{E}}}

\def\one{\mathbbm{1}}
\def\deg{\operatorname{deg}}

\def\Var{\operatorname{Var}}
\def\Cov{\operatorname{Cov}}

\def\Prob{\mathbb{P}}
\def\Pr{\mathbb{P}}

\def\ran{\operatorname{ran}}

\def\Reals{{\mathbb{R}}}

\def\Bin{\operatorname{Bin}}
\def\swap{\operatorname{S}}
\def\Tran{\boldsymbol{T}}%random trees
\def\Tset {\mathcal{T}}                               %tree  sets

\newcommand{\aut}{\textsc{Aut}}
\newcommand{\Aut}[1]{\textsc{Aut}\left(#1\right)}
\newcommand{\Autr}[1]{\textsc{Aut}_r\left(#1\right)}
\newcommand{\Autrabs}[1]{\left|\Autr{#1}\right|}
\newcommand{\Autabs}[1]{\left|\textsc{Aut}\left(#1\right)\right|}
\newcommand{\Autsmall}[0]{\textsc{Aut}_{\rm{small}}}

\newcommand{\Bsmall}{\mathcal{B}_{\rm{small}}}

\newcommand{\prob}{\mathbb{P}}
\newcommand{\Pro}[1]{\prob\left(#1\right)}

\newcommand{\expec}{\mathbb{E}}
\newcommand{\Exp}[1]{\expec\left[#1\right]}

\newcommand{\Vari}[1]{\textup{Var}\left[#1\right]}

\newcommand{\Cova}[1]{{\operatorname{Cov}\left(#1\right)}}

\title{Distribution of  tree parameters  by martingale approach \thanks{The first author's research is  supported by Australian Research Council Discovery Project DP190100977.
and  by Australian Research  Council  Discovery Early Career Researcher Award DE200101045.
The third author's research is supported by the grant 20-31-70025 of Russian Foundation for Basic Research. 
}}

\author{
Mikhail Isaev  \qquad Angus Southwell\\
\small School of Mathematics\\[-0.8ex]
\small Monash University\\[-0.8ex]
\small Clayton, VIC, Australia\\[-0.8ex]
\small\tt \qquad mikhail.isaev@monash.edu   \qquad  angus.southwell@monash.edu
\and
Maksim Zhukovskii\\
\small Laboratory of Combinatorial and Geometric
Structures \\[-0.8ex]
\small Moscow Institute of Physics and Technology\\[-0.8ex]
\small  Dolgoprudny,  Moscow Region, Russian Federation;\\
\small Caucasus mathematical center, Adyghe State University\\[-0.8ex]
\small Maykop,   Republic of Adygea, Russian Federation;\\
\small   The Russian Presidential Academy of National Economy and Public Administration\\[-0.8ex]
 \small   Moscow, Russian Federation;\\
\small\tt zhukmax@gmail.com 
}
\date{}

\begin{document}

\maketitle

\abstract{
 For a uniform random labelled tree, we find the limiting
distribution of tree parameters which are stable (in some sense) with respect to
local perturbations of the tree structure.  The proof is based  on the martingale central limit theorem and the Aldous--Broder algorithm.   In particular, our general result implies the asymptotic normality of the number of occurrences of any given small pattern and the asymptotic log-normality of the number of automorphisms.
}

\section{Introduction}

The distribution of various random variables associated with trees is widely studied in the literature. Typically, the tree  parameters that behave additively exhibit normal distribution, which was observed by
Drmota~\cite[Chapter 3]{Drmota2009}, Janson~\cite{Janson2016}, and Wagner~\cite{Wagner2014}. 
For example, the number of leaves or, more generally, the number of vertices of a given degree 
 satisfies a Central Limit Theorem (CLT)  for many random models: labelled trees, unlabelled trees, plane trees, forests; see  Drmota and Gitteberger~\cite{DG1999} and references therein for more details.

The classical limit theorems of  probability theory are impractical for random trees due to the dependency of adjacencies. Instead, one employs more elaborate tools  such as  
 the analysis of generating functions  \cite{BR1983},  the conditional limit theorems  \cite{Holst1981}, and  Hwang's quasi-power theorem  \cite{Hwang1998}.  
 These methods  are particularly efficient for parameters that admit  a recurrence relation, which is often the case for trees.

 The martingale  CLT \cite{Brown1971} is a powerful tool
  that has  been extensively used to study random structures. 
   Nevertheless,  it is surprisingly overlooked  in the context of the  distribution of tree parameters 
   and   the vast majority of known results rely on the methods mentioned in the paragraph above.    We are aware of only a few applications of  the martingale  CLT:
  Smythe \cite{Smythe1996} and  Mahmoud  \cite{Mahmoud2003} analysed growth of leaves  in the random trees  related to urn models;    M\'ori  \cite{Mori2005}  examined the max degree for Barab\'asi--Albert random trees;     Fen and Hu \cite{FH2011} considered the Zagreb index for random recursive trees;  Sulzbach \cite{Sulzbach2017}   studied the path length in a random model encapsulating binary search trees, recursive trees and plane-oriented recursive trees.

   %and  Cooper, McGrae,  Zito,  \cite{CMZ2009} proved a seven-point concentration result for the empire chromatic
%number of uniform random labelled trees.

 %In this paper we restrict our attention to a simple random model       of unrooted labelled random trees     

 We  prove a CLT for an arbitrary tree parameter using the martingale approach. Unlike other methods, the parameter is  not required to be of a specific form or to satisfy a recurrence relation.  Our only assumption is that the parameter is stable with respect to small perturbations  in the sense that precisely specified below.    We also  bound the rate of convergence  to the normal distribution.   In this paper we restrict our attention to  unrooted  labelled trees even though martingales  appear naturally in many other random settings.   This is  sufficient to demonstrate the power of the new approach and cover several important applications that go beyond the toolkit of existing  methods.

 Let  $\Tset_n$ be the  set of  trees whose vertices are labelled by  $[n]:=\{1,\ldots, n\}$ and $\Tran$ be a uniform random element of $\Tset_n$.  By Cayley's formula, we have $|\Tset_n| = n^{n-2}$. 
 For a tree $T \in \Tset_n$
  and two vertices $i,j\in[n]$, let 
 $d_T(i,j)$ denote the distance between   $i$ and $j $ that is the number of edges in the unique path from $i$ to $j$ in 
 $T$.  For $A,B \subseteq [n]$, let  
  \[d_T(A,B) :=  \min_{u \in A,v \in B} d_T(u,v).\]
% 
% 
%  and let
% \[
% 	\text{$[i,j]_T :=$ the unique path from $i$ to  $j$ in the tree $T$.}
% \]
 Throughout the paper we identify graphs and their edge sets.  Consider an operation defined 
 $\swap_{i}^{jk}$ as follows.
  If $ij \in T$ and $ik \notin T$, let  $\swap_{i}^{jk} T$ be the graph  obtained from $T$ by deleting the edge $ij$ and  inserting
the edge $ik$; see Figure \ref{fig_pert} below. 
% 
% 
% 
% the tree perturbation operator $\swap_{i}^{jk}$ by
%\begin{equation*}
%	\swap_{i}^{jk} T -  ik   = T - ij.
%\end{equation*}
%In other words, $\swap_{i}^{jk} T$ is  obtained from $T$ by deleting the edge $ij$ and  inserting
%the edge $ik$; see Figure \ref{fig_pert} below.
 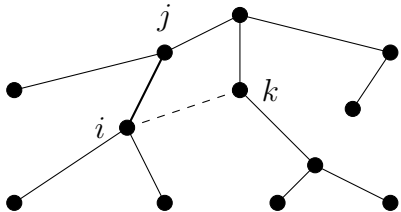
\begin{figure}[h!]
\centering
\begin{tikzpicture}[x=1cm,y=0.5cm]
\node (i) [label=left:{$i$}] at (1.5,2) {};
\node (j) [label=above:{$j$}] at (2,4) {};
\node (k) [label=right:{$k$}] at (3,3) {};

\node (a) at (4.5,0) {};
\node (b) at (4.2,1.5) {};
\node (c) at (5,4) {};

\node (x) at (3,5) {};
\node (y) at (0,3) {};
\node (z)  at (5.5,2.5) {};
\node (t)  at (0,0) {};
\node (u)  at (2,0) {};
\node (v)  at (3.5,0) {};
\node (w)  at (5.7,0) {};

\draw [fill] (w) circle (0.1cm);
\draw [-, thick] (z.center)--(w.center);

\draw [-, thick] (i.center)--(j.center);
\draw[-, dashed,] (i.center) -- (k.center);

\draw [-] (b.center)--(a.center);

\draw [-,] (x.center) -- (j.center)--(y.center);
\draw [-] (t.center) -- (i.center)--(u.center);

\draw [-] (x.center) -- (k.center)--(b.center)--(v.center);
\draw [-] (x.center) -- (c.center)--(z.center);

\draw [fill] (i) circle (0.1cm);
 \draw [fill] (j) circle (0.1cm);  \draw [fill] (k) circle (0.1cm);   \draw [fill] (a) circle (0.1cm);   \draw [fill] (b) circle (0.1cm);\draw [fill] (c) circle (0.1cm); \draw [fill] (x) circle (0.1cm); \draw [fill] (y) circle (0.1cm); \draw [fill] (z) circle (0.1cm); \draw [fill] (t) circle (0.1cm); \draw [fill] (u) circle (0.1cm); \draw [fill] (v) circle (0.1cm);

 %\node   at (5.5,1.5)  {$\Longleftrightarrow$};

\end{tikzpicture}

\smallskip

\caption{$\swap_{i}^{jk}$ removes $ij$ from a tree and adds  $ik$ (dashed)
} \label{fig_pert}
\end{figure}

\noindent
Observe that $\swap_{i}^{jk}T$ is a tree
 if and only if   the path from $j$ to $k$ in $T$ does not contain the vertex $i$.  
We refer the operation  $\swap_{i}^{jk}$ as  a \emph{tree perturbation}.

 %In the following, such 
%triples $(i,j,k)$ are called  \emph{valid} for $T$ or, simply, \emph{$T$-valid}.  

Let $\Reals^+$  denote the set of non-negative real numbers.  %and $\Naturals = \{0,1,\ldots\}$.
For    $\alpha \in \Reals^+$, we say a tree parameter  $F: \mathcal T_n \rightarrow \Reals$ is  \emph{$\alpha$-Lipschitz}  if  
\begin{equation*} % \label{def_Lipschitz}
 	|F(T) - F(\swap_{i}^{jk} T) | \leq \alpha. 
\end{equation*}
for   all $T \in \Tset_n$  and  triples $(i,j,k)$ that $\swap_{i}^{jk} T$ is a tree. 
%Our first result gives a large deviation bound for $F(\Tran)$ when a tree parameter $F$ satisfies the above property.
% and 
% \[
% 	r_F=   \max_{T \in \Tset_n} |F(T) - \E F(\Tran)|.
% \]
% 
% 
% \begin{thm}\label{T:con}
%Let $F: \mathcal T_n \rightarrow \Reals$ be a symmetric tree parameter. If $F$ is $\alpha$-Lipschitz 
%for some   $\alpha \in \Reals^+$, then
%\[
%	\Pr\left(|F(\Tran) - \E F(\Tran)| \geq t \right)
%	 \leq 2 e^{-\frac{2t^2}{n \alpha^2}}
%\]
%\end{thm}
%To get the asymptotic distribution of $F(\Tran)$, 
We also require that the effects on the parameter  $F$  of  sufficiently distant perturbations  $\swap_{i}^{jk}$ and 
 $\swap_{a}^{bc}$   \emph{superpose}; that is 
  \begin{equation*} %\label{def_super}
 	F(\swap_{i}^{jk}\swap_{a}^{bc}T)	 - F(T) =
   \left(F(\swap_{i}^{jk}T) - F(T)\right) + 
   \left(F(\swap_{a}^{bc}T) - F(T)\right).
 \end{equation*}
  For   $\rho \in \Reals^+$,  we say  $F$ is \emph{$\rho$-superposable}  if 
%  \begin{equation}\label{def_super}
%  	|F(T) - F(\swap_{i}^{jk} T) - F(\swap_{a}^{bc} T )  + F(\swap_{i}^{jk}  \swap_{a}^{bc} T ) | \leq \rho(d_T(\{j,k\},\{b,c\}))
%  \end{equation} 
the above equation  holds  for all $T \in \Tset_n$   and  triples $(i,j,k)$, $(a,b,c)$ such that 
  $\swap_{i}^{jk}T$, $\swap_{a}^{bc}T$,  $\swap_{i}^{jk}\swap_{a}^{bc}T$ are trees and  
  $d_T(\{j,k\}, \{b,c\})\geq  \rho$.
     Note that   the sets $\{j,k\}$ and $\{b,c\}$ are at the  same distance in all four trees
$T$, $\swap_{i}^{jk}T$, $\swap_{a}^{bc}T$, and $\swap_{i}^{jk}\swap_{a}^{bc}T$. Thus,
     $d_T(\{j,k\}, \{b,c\})$ is an appropriate measure for the 
  distance between  the two tree perturbations  $\swap_{i}^{jk}$ and  $\swap_{a}^{bc}$.

  For a random variable $X$ let
\[
	\delta_{\rm K} \left[X\right] := \sup_{t \in \Reals} \left| \Pr\left(X  - \E[X]  \leq t (\Var [X])^{1/2} \right) - \Phi(t)\right|,
\] 
where 
 $\Phi(t) =  (2\pi)^{-1/2} \int_{-\infty}^t e^{-x^2/2} dx$.  In other words, $\delta_{\rm K} [X]$ is the Kolmogorov distance between the scaled random variable $X$ and the standard normal distribution.  We say $X=X_n$ is \emph{asymptotically normal}  if $\delta_{\rm K} [X]  \rightarrow 0$ as $n \rightarrow \infty$.

 In the following theorem,  $F$, $\alpha$, and $\rho$ 
 stand for sequences parametrised by a positive integer $n$ that is  $(F,\alpha, \rho)  = (F_n, \alpha_n, \rho_n)$. 
We omit the subscripts  for notation simplicity.   %Let   $\Tran$ be a uniform random element of $\Tset_n$. 
 All asymptotics in the paper refer to $n\rightarrow \infty$  and the notations $o(\cdot)$, $O(\cdot)$, $\Theta(\cdot)$ have the standard meaning.
 
\begin{thm}\label{T:main}
 Let  a tree parameter $F: \mathcal T_n \rightarrow \Reals$ 
   be $\alpha$-Lipschitz and  $\rho$-superposable for some $\alpha> 0$ and $\rho\geq 1$.
   Assume also that,      for a  fixed constant $\varepsilon>0$,
  \[
  	  \frac{n \alpha^3}{(\Vari{F(\Tran)})^{3/2}} + \frac{n^{1/4}\alpha \rho}{(\Vari{F(\Tran)})^{1/2}} 
  	 =  O(n^{-\varepsilon}).
  \]
 Then,  $F(\Tran)$ is asymptotically normal. Moreover, 
  %  $\dfrac{F(\Tran) - \E F(\Tran)}{\sigma}$ converges in distribution to $\mathcal{N}(0,1)$
   $\delta_{\rm K}[F(\Tran)] = O(n^{-\varepsilon'})$  	for any $\varepsilon' \in (0,\varepsilon)$.
\end{thm}

To clarify the assumptions \cref{T:main},  we consider a simple application to the aforementioned parameter   $L(T)$,   the number of leaves in a tree $T$. 
The distribution of $L(\Tran)$ was derived for the first time by Kolchin \cite{Kolchin1977}, using generating functions and the connection to the Galton--Watson branching process. \cref{T:main} immediately leads to the following result:
\begin{cor}\label{c:leaves}
	$L(\Tran)$ is asymptotically normal and  $\delta_{\rm K}[L(\Tran)] = O(n^{-1/4+\epsilon})$
	for any $\epsilon>0$.
%	and 
%	$
%	 	\Pr\left(|L(\Tran) - \E L(\Tran)| \geq t \right)
%	 \leq 2 e^{-\frac{2t^2}{n}}.
%	$
\end{cor}
\begin{proof}
For any tree $T\in \Tset_n$ and a triple $(i,j,k)$ that  $\swap_{i}^{jk}T$ is a tree,
the numbers of leaves of $T$ and $\swap_{i}^{jk}T$ differ by at most one. 
Thus, $L$ is $\alpha$-Lipchitz on $\Tset_n$ with $\alpha=1$. 
%Applying \cref{T:con}, we get the 
%required bound for $\Pr\left(|L(\Tran) - \E L(\Tran)| \geq t \right)$.

Next, observe that if 
$T$, $\swap_{i}^{jk} T$, $\swap_{a}^{bc} T$, and   $\swap_{i}^{jk}  \swap_{a}^{bc} T$
are trees and  $\{j,k\}\cap \{b,c\} = \emptyset$, then 
\[
	L(T) - L(\swap_{i}^{jk} T) - L(\swap_{a}^{bc} T )  + L(\swap_{i}^{jk}  \swap_{a}^{bc} T ) =0.
\]
Indeed, the trees $T$, $\swap_{i}^{jk} T$, $\swap_{a}^{bc} T$, $\swap_{i}^{jk}  \swap_{a}^{bc} T$ have the same sets of leaves except possibly vertices $\{j,k,b,c\}$.   However,  any vertex from $\{j,k,b,c\}$  contributes to the same number of negative and positive terms in the  left-hand side of the above. This  implies that   $L$ is $\rho$-superposable with $\rho=1$.

 It is well known that $\Var [L(\Tran)] =(1+o(1))n/e$; see, for example, \cite[Theorem 7.7]{Moon1970}. Then,  all the assumptions of \cref{T:main} are satisfied  with 
 $ \alpha = \rho = 1$ and $\varepsilon = 1/4$.
 This completes the proof.
\end{proof}

\begin{remark}
 The rates of convergence $\delta_K[F(\Tran)]= O(n^{-1/4+\epsilon})$ are typical in applications of \cref{T:main} because,	for many examples,  $\Var [F(\Tran)]$  is linear   and
$\alpha$, $\rho$ are bounded by some power of $\log n$.  
Wagner \cite{Wagner2019} pointed out that 
Hwang's quasi-power theorem  \cite{Hwang1998} leads to a better estimate  $\delta_K[L(\Tran)]= O(n^{-1/2+\epsilon})$ for the number of leaves. This matches the rates of convergence in the classical Berry--Esseen result (for a sum of i.i.d. variables) and, thus,  is likely optimal. It remains an open question whether the bound $\delta_K[F(\Tran)]= O(n^{-1/2+\epsilon})$
always hold for an arbitrary  $\alpha$-Lipschitz and   $\rho$-superposable  tree parameter  $F$ 
(assuming the variance is linear and  $\alpha$ and $\rho$ are not too large).
\end{remark}

The asymptotic normality of the number of vertices   in $\Tran$ with a given degree 
is proved identically to \cref{c:leaves}.   However, for many other applications, a tree parameter  $F$ might behave badly on a small set of trees.  Then, \cref{T:main} does not work directly since $\alpha$ and $\rho$ are too large. For example, a single perturbation 
 $\swap_{i}^{jk}$ can destroy a lot of  paths on three vertices in a tree with large degrees.
 To overcome this difficulty, one can apply  \cref{T:main} to  a parameter  $\tilde{F}$, which is related to $F$, but ignores the vertices with degrees larger $\log n$. This trick does not change the limiting distribution   because the trees with large degrees are rare:  Moon \cite[formula (7.3)]{Moon1970} showed  that, for any $d \in [n]$,
\begin{equation}\label{eq:Moon} 
 \Pr(\Tran \text{ has a vertex with degree $>d$} ) \leq n/d! 
 \end{equation} 
       Similarly, one can restrict attention to the trees 
 for which the neighbourhoods of  vertices do not grow very fast.  
 Let 
 \begin{equation}\label{def_beta}
 	\beta(T) = \max_{i,d \in [n]} \frac{|\{j \in[n] \st d_T(i,j) = d\}|}
 	   {d}.
 \end{equation}
 In this paper, we prove the following result, which might be of independent interest.
 \begin{thm}\label{T:beta}
 $ \displaystyle
   	\Pr \left( \beta(\Tran) \geq \log^4 n \right)  \leq e^{-\omega(\log n)}.
   	$
 \end{thm}
% For convenience, further in the paper we allow
% $e^{-\omega(\log n)}$ to be equal $0$. In particular,
% $a_n = (1+e^{-\omega(\log n)}) b_n$ means 
% that  $(a_n  b_{n}^{-1}  -1) n^{c} \rightarrow 0$ for any constant $c$.
\begin{remark} 
The distribution of  the height profiles in branching processes is a well-studied topic.
 In particular, the number of vertices in $\Tran$ at distance  at most $d$ from a given vertex
 was already considered by Kolchin \cite{Kolchin1977}. However, we could not find a suitable large deviation   bound for $\beta(\Tran)$  in the literature. In fact, the constant $4$ in the exponent of the logarithm in the bound above  is not optimal, but sufficient for our purposes.     
\end{remark}

 The structure of the paper is as follows. 
 In \cref{S:patterns} we analyse the number of occurrences of an arbitrary tree pattern.  
For various interpretations of  the notion ``occurrence'',  the asymptotic normality in this problem was established by Chysak, Drmota, Klausner, Kok \cite{CDKK2008} and  Janson~\cite{Janson2016}.
Applying  \cref{T:main}, we not only confirm these results, but also allow 
much more general types of occurrences. In particular, we prove the 
asymptotical normality for  the number of induced subgraphs isomorphic to a given  tree of fixed  size 
and  for the number of paths of length  up to $n^{1/8 -\eps}$.  Both of these applications  go beyond the setup of  \cite{CDKK2008, Janson2016}.
In \cref{S:auto} we derive  the distribution of the number of automorphisms of $\Tran$ and  confirm the conjecture  by Yu \cite{Yu2012}. To our knowledge, this  application of \cref{T:main} is  also not covered by any of the previous results.

We prove  \cref{T:main} in  \cref{S:Construction}, using a martingale construction based on 
the Aldous--Broder algorithm \cite{Aldous1990} for generating  random labelled spanning trees of a given graph. \cref{S:m_theory} contains the necessary background on the theory of martingales.
We also use martingales to prove \cref{T:beta} in \cref{S:balls}. 
This proof is independent of   \cref{S:Construction} and, in fact,  \cref{T:beta} is one of the ingredients that 
we need for our main result, \cref{T:main}.   
We also use \cref{T:beta} in the application to long induced paths to
 bound the number of the paths  affected by one perturbation;  see \cref{T:CLT_paths}. 

Tedious technical calculations of the variance for the pattern  and automorphism  counts
are given in Appendix A and  Appendix B.

%%%%%%%%%%%%%%%%%%%%%%%%%%%%%%%%%%%%%%
%%%%%%%%%%%%%%%%%%%%%%%%%%%%%%%%%%%%%%
%%%%%%%%%%%%%%%%%%%%%%%%%%%%%%%%%%%%%%
%%%%%%%%%%%%%%%%%%%%%%%%%%%%%%%%%%%%%%
%%%%%%%%%%%%%%%%%%%%%%%%%%%%%%%%%%%%%%%%%%%%%%%%%%%%%%%%%%%%
%%%%%%%%%%%%%%%
\section{Pattern counts}\label{S:patterns}

In this section we apply \cref{T:main} to analyse the distribution of the number of  occurrences of a tree pattern $H$
as an induced subtree in uniform random labelled tree $\Tran$.
To our knowledge,  the strongest results for this problem were obtained by  Chysak, Drmota, Klausner, Kok \cite{CDKK2008} and  Janson~\cite{Janson2016}.

Chysak, Drmota, Klausner, Kok \cite{CDKK2008} consider occurrences of a pattern $H$  as an induced subgraph of  a tree $T$ with  the additional restriction that  the \emph{internal} vertices in the pattern match the degrees the corresponding vertices in $T$. That is, the other edges of $T$ can only be adjacent to leaves of $H$.  For example, the tree $T$ on  Figure \ref{F:patterns} 
contains only three paths on three vertices in this sense, namely $T[\{1,5,8\}]$, $T[\{1,3,6\}]$,
and $T[\{3,6,13\}]$. 
In particular, the induced path on vertices $1$, $2$, $7$ is not counted since the internal vertex $2$ is adjacent to  $4$. 
The result by Chysak, Drmota, Klausner, Kok is given below.

 \begin{figure}[h!]
\centering
\begin{tikzpicture}[x=1cm,y=0.5cm]
\node (i)[label=left:{$7$}]  at (1.5,2) {};
\node (j)[label=$2$]   at (2,4) {};
\node (k)[label=left:{$5$}]  at (3,3) {};

\node (a)[label={[shift={(0.2,-0.2)}] $12$}]  at (4.5,0) {};
\node (b)[label={[shift={(0.1,-0.2)}]$8$}]  at (4.2,1.5) {};
\node (c)[label=right:{$3$}]  at (5,4) {};

\node (x)[label={[shift={(0.1,-0.2)}]$1$}]   at (3,5) {};
\node (y)[label=$4$] at (0,3) {};
\node (z)[label=right:{$6$}]  at (5.5,2.5) {};
\node (t)[label=$9$]  at (0,0) {};
\node (u)[label={[shift={(0.2,-0.2)}] $10$}]   at (2,0) {};
\node (v)[label={[shift={(-0.2,-0.2)}] $11$}]   at (3.5,0) {};

\node (w)[label={[shift={(0.2,-0.2)}] $13$}]  at (5.7,0) {};

\draw [fill] (w) circle (0.1cm);
\draw [-, thick] (z.center)--(w.center);

\node[above] at (3,6.5) {Tree $T$};

\draw [-, thick] (i.center)--(j.center);
%\draw[-, dashed,] (i.center) -- (k.center);

\draw [-] (b.center)--(a.center);

\draw [-,] (x.center) -- (j.center)--(y.center);
\draw [-] (t.center) -- (i.center)--(u.center);

\draw [-] (x.center) -- (k.center)--(b.center)--(v.center);
\draw [-] (x.center) -- (c.center)--(z.center);

\draw [fill] (i) circle (0.1cm);
 \draw [fill] (j) circle (0.1cm);  \draw [fill] (k) circle (0.1cm);   \draw [fill] (a) circle (0.1cm);   \draw [fill] (b) circle (0.1cm);\draw [fill] (c) circle (0.1cm); \draw [fill] (x) circle (0.1cm); \draw [fill] (y) circle (0.1cm); \draw [fill] (z) circle (0.1cm); \draw [fill] (t) circle (0.1cm); \draw [fill] (u) circle (0.1cm); \draw [fill] (v) circle (0.1cm);

 %\node   at (5.5,1.5)  {$\Longleftrightarrow$};
\begin{scope}[shift={(10,0)}]
		\node (x) at (0,5) {};
		\node (k)  at (0,2.5) {};
		\node (l)  at (0,0) {};
		\draw[-] (x.center) -- (k.center) --(l.center);
		%\draw[vert] (0,0) circle [radius=0.1];
		%\draw[vert2] (0,1) circle [radius=0.1];
		\node[above] at (0,6.5) {Pattern $H$};
		 \draw [fill] (x) circle (0.1cm);  \draw [fill] (k) circle (0.1cm);  \draw [fill] (l) circle (0.1cm);
		\end{scope}
\end{tikzpicture}

\medskip

\caption{A labelled tree $T$ and a pattern $H$
} \label{F:patterns}
\end{figure}
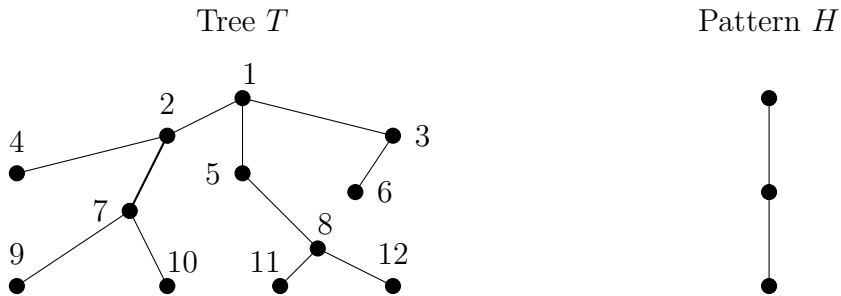

\begin{thm}{(\cite[Theorem 1]{CDKK2008})} \label{T:Drmota}
   Let $H$ to be a given finite tree. Then the limiting distribution of the number of occurrences of $H$ (in the sense described above) 
 in $\Tran$ is asymptotically normal with mean and variance asymptotically equivalent to $\mu n$ and $\sigma^2 n$,
   where $\mu>0$ and $\sigma^2\geq 0$ depend on the pattern $H$ and can be computed explicitly 
   and  algorithmically and can be represented as polynomials (with rational coefficients) in $1/e$.
\end{thm}

Janson~\cite{Janson2016} considers the subtree counts $\eta_H(T)$  defined differently.  Set  vertex $1$ to be a root of $T\in \Tset_n$. For any other vertex $v$, let $T_v$  be the subtree consisting of $v$ an all its descendants. Such subtrees are called \emph{fringe subtrees}. The parameter $\eta_H(T)$ equals the number of fringe subtrees isomorphic to $H$ (with a root).  For example,
the tree $T$ on  Figure \ref{F:patterns} 
contains only one path with three vertices (rooted at end vertex), namely $T[\{3,6,13\}]$.  
 In particular, the induced paths  $T[\{1,5,8\}]$ and $T[\{1,3,6\}]$ are not counted since they are not  fringe subtrees. Janson~\cite{Janson2016} proved  the following result about joint asymptotic normality for several such subtree counts.
\begin{thm}{(\cite[Corollary 1.8]{Janson2016})} \label{Janson2016}
Let $\Tran^{\rm GW}_n$ be a conditioned Galton--Watson tree of order $n$ with offspring distribution $\xi$,
 where $\E[\xi] = 1$ and $0 < \sigma^2 := \Var [\xi] <\infty$. Then, the subtree counts $\eta_H(\Tran^{\rm GW}_n)$ 
 (for all $H$ from a given set of patterns)  are asymptotically jointly normal.
\end{thm}
 Janson~\cite[Corollary 1.8]{Janson2016} also gives expressions for the covariances of the limiting distribution  in terms of  
 the distribution of the corresponding unconditioned Galton--Watson tree. To relate this model to uniform random labelled tree $\T$, one need to take  the conditioned Galton--Watson tree of order $n$ with  the Poisson offspring distribution.

We consider a more general type of tree counts which encapsulates both counts from above. 
In fact, it was suggested by Chysak, Drmota, Klausner, Kok \cite{CDKK2008}:
 ``\emph{...we could also consider pattern-matching problems for patterns in which some degrees of certain
possibly external ``filled'' nodes must match exactly while the degrees of the other, possibly internal ``empty'' nodes
might be different. But then the situation is more involved.}'' Then, in  \cite[Section 5.3]{CDKK2008}  they explain that
having an internal    ``empty'' node leads to serious complications in their approach.

We define our tree parameter formally.  Let  $H$ be a tree with $\ell$ vertices $v_1,\ldots,v_{\ell}$.      
Let $\thetavec = (\theta_1,\ldots,\theta_\ell) \in \{0,1\}^{\ell}$. 
We say the pattern $(H,\thetavec)$  \emph{occurs}  
in a tree $T \in \Tset_n$  if 
there exists a pair of sets $(U,W)$ such that $W\subset U \subset [n]$ and
\begin{itemize}
\item the induced 
subgraph $T[U]$ is isomorphic to $H$, 
\item the set $W$ 
corresponds to all vertices $v_i$ with $\theta_i=1$ (``empty'' nodes), 

\item  there is no edge in $T$ between  $U-W$  and $[n] - U$.
 \end{itemize}
 Denote by  $N_{H,\thetavec}(T)$ the number of occurrences of the pattern $(H,\thetavec)$ in $T$ that is the number of different pairs  $(U, W)$ satisfying the above. It equals the number 
 of ways  to choose suitable identities for $v_1,\ldots, v_\ell$ in $[n]$  divided 
 by  $|\Aut{H,\thetavec}|$, the number of automorphisms of $H$ that preserve  $\thetavec$. In particular, 
 if $\theta_i = 1$ for all $i\in [\ell]$ then $N_{H,\thetavec}(T)$ is the number of induced subgraphs in $T$ isomorphic $H$.
 If $\theta_i = 1$ whenever $i$ is a leaf of $H$, then $N_{H,\thetavec}(T)$ is the tree count  considered in \cref{T:Drmota}. If $\theta_i=1$ for exactly one vertex $i\in [\ell]$ which is a leaf in $H$, then $N_{H,\thetavec}(T)$  counts fringe subtrees.

In \cref{S:distr_patt}, we prove that  $ N_{H,\thetavec}(\Tran)$ is asymptotically normal for any fixed $H$ and $\thetavec\in \{0,1\}^{\ell}$ with at least one non-zero component (where $\ell$ is the number of vertices in $H$). Note  that  if  $\theta_i=0$ for all $i \in [\ell]$ and $n>\ell$ then $N_{H,\thetavec}(\Tran) = 0$
since at least one vertex corresponding to $H$ must be adjacent to other vertices in $T$.
 Our approach also works for growing patterns. We demonstrate it for the case when  $H$ is a path.

\subsection{Moments calculation}

To apply \cref{T:main}, we need a lower bound for $\Var(N_{H,\thetavec}(\Tran))$.  
One can  compute the moments of  $N_{H,\thetavec}(\Tran)$  using the following formula
for the number of trees  containing a given spanning forest. 
Lemma \ref{L:forest}  is a straightforward generalisation of \cite[Theorem 6.1]{Moon1970} with   almost identical proof, which we include here for the sake of completeness.

\begin{lemma}\label{L:forest}
	Let $S = H_1 \sqcup\ldots\sqcup H_k$ be a forest on $[n]$ and 
	$B_i$ be  non-empty subsets (not necessarily proper) of $V(H_i)$ for all 
	$i \in [k]$. Then, the number of trees  $T$
	on $[n]$ containing all edges of $H$ such that $\deg_T(v) = \deg_S(v)$
	for every $v$ outside $B_1 \cup \ldots \cup B_{k}$
	equals $b_1\cdots b_k (b_1+\cdots + b_k)^{k-2}$,
	where $b_i$ is the number of vertices in $B_i$.
\end{lemma}

\begin{proof}
     Any desired tree $T$ corresponds to a tree  $T_H$  on $k$ vertices labelled by $H_1,\ldots, H_k$ for which  the  vertices $H_i$ and $H_j$ are adjacent if and only if there is an edge between $H_i$ and
      $H_j$ in $T$. 
      If $d_1,\ldots, d_k$ are degrees of $T_H$ then the number of trees $T$ corresponding to $T_H$
      equals $b_1^{d_1} \ldots b_{k}^{d_{k}}$ since we can only use vertices from $B_1 \cup \ldots \cup B_{k}$ for edges of $T$.  From 
\cite[Theorem 3.1]{Moon1970}, we know that the number of trees
on $k$ vertices with degrees $d_1,\ldots, d_k$ is 
$
\binom{k-2}{d_1-1, \ldots ,d_k-1}.
$
Thus, the total number of such trees $T$  is 
\[
	\sum_{(d_1,\ldots, d_k)} 
	b_1^{d_1} \ldots b_{k}^{d_{k}}
	\binom{k-2}{d_1-1, \ldots ,d_k-1} = b_1\ldots b_k (b_1+\cdots +b_k)^{k-2},
\]
where the sum is over all positive integers sequences  that $d_1+\cdots+ d_k = 2k-2$.
\end{proof}

For an $\ell$-tuple $\uvec= (u_1,\ldots,u_\ell) \in [n]^{\ell}$  with distinct coordinates, let 
$\one_{\uvec}(\Tran)$  be the indicator of the event that  a pattern $(H,\thetavec)$ occurs  in $\Tran$  with  $u_1, \ldots, u_\ell$ corresponding  to the vertices of $H$.
Let $s: = \sum_{i=1}^{\ell} \theta_i$.
 Applying Lemma \ref{L:forest} to a forest consisting of one nontrivial component isomorphic to $H$ and  dividing by $|\Tset_n| = n^{n-2}$, we find that 
\begin{equation}\label{e:patterns}
	\E \left[ \one_{\uvec}(\Tran)\right] = 
	\frac{s  \left(n -  \ell  +  s  \right)^{n-\ell -1}}
	{n^{n-2}}
	=  \frac{s e^{-\ell + s+  O(\ell^2/n) }}{ n^{\ell-1}}.
\end{equation}
Summing over all choices for $\uvec$ and dividing by  $|\Aut{H,\thetavec}|$ to adjust overcounting, we get 
\begin{align*}
	\E  \left[N_{H,\thetavec}(\Tran) \right]&=
	 \frac{1}{|\Aut{H,\thetavec}|}  \sum_{\uvec} \E \left[ \one_{\uvec}(\Tran)\right]
	=n \, \frac{s   e^{-\ell +s +  O(\ell^2/n) } }{\Aut{H,\thetavec}}  
	.
\end{align*}
In particular, this formula agrees with \cref{T:Drmota} that $\mu$ is a polynomial with rational coefficients in $1/e$. Similarly, for the variance, we have 
\begin{equation}\label{v:patterns}
	\Var   \left[N_{H,\thetavec}(\Tran)\right]
	        =  \frac{1}{|\Aut{H,\thetavec}|^2} 
	        \sum_{\uvec,\uvec'} \Cov (\one_{\uvec}(\Tran), \one_{\uvec'}(\Tran) ),
\end{equation}
where the sum over all $\ell$-tuples $\uvec,\uvec' \in [n]^{\ell}$  with distinct coordinates.
Then, we can also use  Lemma \ref{L:forest} (with one or two nontrivial components) to compute
$ \Cov ( \one_{\uvec}(\Tran), \one_{\uvec'}(\Tran))$. However, this computation is much more involved: one needs to consider all possible ways the pattern $(H,\thetavec)$  intersects with itself. Nevertheless, for a fixed pattern,  it is not difficult to see that 
$\E \left[\one_{\uvec}(\Tran)\right]$ and 
$\E \left[\one_{\uvec}(\Tran) \one_{\uvec'}(\Tran)\right]$  are polynomials with integer coefficients in 
$1/e$ divided by some power of $n$. This observation is already sufficient to establish 
the bound $\Vari{N_{H,\thetavec}(\Tran)}= \Omega(n)$ for the case 
when $\sum_{i=1}^{\ell} \theta_i <\ell$. 
 \begin{lemma}\label{L:Var}
 	 Let $(H,\thetavec)$ be a fixed pattern, $\ell$ be the number of vertices in tree $H$, and $s:=\sum_{i=1}^{\ell} \theta_il$. Then, there exist a polynomial 
 	 $p_{H,\thetavec}$ of degree at most $2\ell-2s$  with integer coefficients  that
 	 \[
 	 \Vari{N_{H,\thetavec}(\Tran)} = n \, \frac{p_{H,\thetavec} (1/e)}{|\Aut{H,\thetavec}|^2} + O(1).
 	 \]
 	 Moreover, if  $s <\ell$ then
 	 $p_{H,\thetavec} (1/e) > 0$.
 \end{lemma}
 \begin{proof}
 	Consider any $\ell$-tuples $\uvec,\uvec' \in [n]^{\ell}$  with distinct coordinates.   
 	If the coordinates of $\uvec$ and $\uvec'$ form disjoint sets, then 
 	 applying Lemma \ref{L:forest} to a forest consisting of two nontrivial component isomorphic to $H$, we find that 
\begin{equation*}
	\E \left[\one_{\uvec}(\Tran) \one_{\uvec'}(\Tran) \right] = 
	\frac{s^2  \left(n -  2\ell  +  2s  \right)^{n-2\ell}}
	{n^{n-2}}.
\end{equation*}
Using \eqref{e:patterns}, we get that
 \begin{align*}
 	 \Cov ( \one_{\uvec}(\Tran),  \one_{\uvec'}(\Tran)) &= 
 	  \frac{s^2}{n^{2\ell-2}}  \left( e^{-2 \ell + 2s -  \frac{(2\ell-2s)^2}{2n} +O(n^{-2}) }
 	 - e^{2 (- \ell + s -  \frac{(\ell-s)^2}{2n} ) +O(n^{-2})}
 	 \right)\\
    &= - \frac{s^2  (\ell-s)^2  e^{-2 \ell + 2s} }{n^{2\ell-1}} + O(n^{-2\ell}). 
 \end{align*}
 Then, the contribution  of such $\uvec,\uvec'$  to the sum $\sum_{\uvec,\uvec'} \Cov (\one_{\uvec}(\Tran), \one_{\uvec'}(\Tran) )$
 in  \eqref{v:patterns} 
 equals \[-n  s^2  (\ell-s)^2  e^{-2 \ell + 2s} + O(1).\]  
 
 Next, we proceed to the case when the sets formed by the coordinates of $\uvec$ and $\uvec'$ intersect. 
Let $a$ be the size of the union of these two sets and
 \[
 	 b:=|\{u_i \st \theta_i=1\}\cap \{u_i' \st \theta_i=1\}|.
 \]
 Note that  $\ell - s \leq a-b \leq  2\ell -2 s$.
 Then, using Lemma \ref{L:forest} (and also \eqref{e:patterns}), we find that 
 \begin{align*}
 	\Cov(\one_{\uvec}(\Tran),\one_{\uvec'}(\Tran)) 
 	 &= \frac{b(n - a +b)^{n-a-1}}{n^{n-2}} - \left(\frac{s (n-\ell +s)^{n-\ell-1}}{n^{n-2}}\right)^2
        \\
        & =  \frac{1 + O(n^{-1})}{n^{a-1}}  \cdot
        \begin{cases}
        be^{-a+b} , &\text{if } a\leq 2\ell-2,
         \\
         b e^{-a+b} - s^2 e^{-2\ell + 2s}, &\text{if } a=2\ell -1.
        \end{cases}
 \end{align*}
 We say a pair $(\uvec,\uvec')$ is equivalent to $(\wvec,\wvec')$ if there is a permutation $\sigma$ 
 of the set $[n]$ that $w_i = \sigma(u_i)$ and $w_i' =  \sigma(u_i')$ for all $i \in [\ell]$. Note that the number of pairs equivalent to  $(\uvec,\uvec')$ is exactly $(n)_{a}$. Then, the contribution of the equivalence  class   to the sum $\sum_{\uvec,\uvec'} \Cov (\one_{\uvec}(\Tran), \one_{\uvec'}(\Tran) )$
 in  \eqref{v:patterns}  is   $n be^{-a+b} +O(1)$ or $nbe^{-a+b} - s^2 e^{-2\ell +2s}+ O(1)$.  Summing over all equivalence classes, we  complete the proof of the first part.
 
 For the second part, observe in the above that $a-b = \ell -s$ if and only if 
 the sets of coordinates of $\uvec$ and $\uvec'$ coincide and 
 $\{u_i \st \theta_i=1\}=\{u_i' \st \theta_i=1\}$.  In particular,  we have $a<2\ell -1$ so 
 $	\Cov(\one_{\uvec}(\Tran),\one_{\uvec'}(\Tran)) >0$. Then, the coefficient corresponding to $x^{-\ell + s}$ in $p_{H,\theta}(x)$ is strictly positive so the polynomial $p_{H,\theta}(x)$ is not trivial. Since the number $1/e$ is transcendental, we conclude that $p_{H,\theta}(1/e)$ is not zero. Also, $p_{H,\theta}(1/e)$ can not be negative since $\Vari{N_{H,\thetavec}(\Tran)} \geq 0$ so it can only be positive. This completes the proof.
 \end{proof}

 For a tree $T\in \Tset_n$, 
  let  $N_{H}(T) := N_{H, \thetavec}(T)$ if $\theta_i=1$ for all $i \in [\ell]$ that is 
  $N_{H}(T)$ is  the number of induced subgraphs of $T$ isomorphic to $H$.
 Unfortunately,  the lemma above can not guarantee that $\Vari{N_{H}(\Tran)} = \Omega(n)$. In this case,  the polynomial $p_{H,\theta}$ is  a non-negative constant, but an additional argument is required to show that it is not zero.  
  \begin{lemma}\label{l:Var_pattern}
  	For any fixed tree $H$ with degrees $h_1,\ldots,h_\ell$, we have
  	\[
  		\Vari {N_{H}(\Tran)}  \geq  \frac{n}{|\Aut{H}|^2}\sum_{j \geq 2} c_j^2 \, j!  + O(1),
  	\] 
  	where $c_j = \sum_{i=1}^\ell \left( \binom{h_i}{j} + (\ell-1) \binom{h_i-1}{j} \right)$.
  	In particular, $c_2>0$ if $\ell \geq 3$.
  \end{lemma}
   The proof of Lemma \ref{l:Var_pattern} is given in Appendix A of the ArXiv version  \cite{TreesArxiv} of the current paper. The key idea of this proof is to  estimate the variance of the conditional expection value of   $N_{H}(\Tran)$ given the degree sequence of $\Tran$.

  \remark{
  There is a different way to show $\Vari{N_{H,\thetavec}(\Tran)} =\Omega(n)$
  for any fixed $H$  and $\thetavec$ (including the case $\theta_i=1$ for all $i \in [\ell]$). First, one establishes that $\Pr(N_{H,\thetavec}(\Tran)  = x_n) = o(1)$ for any sequence $x_n$. 
  Reducing/incrementing the number of  fringe copies of $H$ in a clever way shows that $\Pr(N_{H,\thetavec}(\Tran)  = x_n)$ is not much larger than   
  $\Pr(N_{H,\thetavec}(\Tran)  = x_n -k) + \Pr(N_{H,\thetavec}(\Tran)  = x_n +k)$
  for all $k$ from a sufficiently large set. This  implies  that 
   $\Vari{N_{H,\thetavec}(\Tran)} \rightarrow \infty$. Therefore, $p_{H,\theta}>0$   so $\Vari {N_{H,\thetavec}(\Tran)} = \Omega(n)$.   In fact, the proof of \cref{l:Var_pattern} 
  given in \cite[Appendix A]{TreesArxiv}  is more technically involved than this idea, but it extends better to  growing substructures.
  }
 
 \medskip
 
Using formula \eqref{v:patterns}, we also obtain a precise estimate of $\Vari{N_{H}(\Tran)}$ for the case when $H$ is a path.  With slight abuse of notations,  let
$P_\ell(T):= N_{P_\ell}(T)$  that is  the number of paths on $\ell$ vertices in a tree $T\in \Tset_n$.

\begin{lemma}\label{l:Varpaths}
  Let $\ell>2$  and $\ell = O(n^{1/2})$, then
  \[
  	\Vari{P_{\ell}(\Tran)} =\left (1+O\left(\dfrac{\ell^2}{n}\right)\right) n\ \frac{\ell(\ell-1)^2 (\ell-2)}{24}.
  \]
\end{lemma}
 \begin{proof}
 	 For the induced path counts
 	formula \eqref{v:patterns} simplifies as follows:
 	\[
	\Vari { P_{\ell}(\Tran)}
	        =  \dfrac{1}{4} 
	        \sum_{\uvec,\uvec'} \Cov (\one_{\uvec}(\Tran), \one_{\uvec'}(\Tran) ).
\]
For $i \in [\ell]$, let $\Sigma_i$ be the set of pairs $(\uvec,\uvec')$ that the sets formed by its coordinates have exactly $i$ elements in common.
From \eqref{e:patterns}, we have that 
$\Exp{ \one_{\uvec}(\Tran)} = \ell n^{1-\ell}$. Using \cref{L:forest}, we get 
$\Exp{\one_{\uvec}(\Tran)\one_{\uvec'}(\Tran)} = \ell^2 n^{2-2\ell}$ for $(\uvec,\uvec') \in \Sigma_0$, so
\[
	\sum_{(\uvec,\uvec')\in \Sigma_0} \Cov (\one_{\uvec}(\Tran), \one_{\uvec'}(\Tran)) = 0.
\]
Applying \cref{L:forest}, it is  a routine to check that 
\begin{align*}
 \sum_{(\uvec,\uvec') \in \Sigma_1} \Cov (\one_{\uvec}(\Tran), \one_{\uvec'}(\Tran) )
  &= |\Sigma_1|  \left( (2\ell -1) n^{2-2\ell} - \ell^2 n^{2-2\ell}\right)  \\&=
    - \frac{(n)_{2\ell-1}}{n^{2\ell-2}} \ell^2  (\ell-1)^2
 = 
     - \left (1+O\left(\dfrac{\ell^2}{n}\right)\right)    n \ell^2  (\ell-1)^2.
    \end{align*}
    Similarly, for $2 \leq i\leq \ell$, we get
    \begin{align*}
    	 \sum_{(\uvec,\uvec') \in \Sigma_i} \Cov (\one_{\uvec}(\Tran), \one_{\uvec'}(\Tran) )
    	 &= |\Sigma_i| \left( (2\ell -i) n^{1-2 \ell +i} - \ell^2 n^{2-2\ell}\right)
    	 \\
    	 &=  \left (1+O\left(\dfrac{\ell}{n}\right) \right)\frac{2 (n)_{2\ell-i}}{n^{2\ell-i-1}} (\ell- i+1)^2 (2\ell-i)
    	 \\
    	 &=\left (1+O\left(\dfrac{\ell^2}{n}\right)\right)    2n (\ell- i+1)^2 (2\ell-i).
    \end{align*}
    Summing the above bounds  for $\Sigma_0,\ldots, \Sigma_{\ell}$
    and using 
    \[
     -\ell^2  (\ell-1)^2 + 2 \sum_{i=2}^{\ell} (\ell- i+1)^2 (2\ell-i) = \frac{\ell(\ell-1)^2 (\ell-2)}{6}, 
    \]   we  get the stated formula  for
     $\Var   P_{\ell}(\Tran) $.
%    get that
%    \[
%    	\Var   P_{\ell}(\Tran) &= \left (1+O\left(\dfrac{\ell^2}{n}\right)\right)  \frac{n}{4} 
%    	\left( -  \ell^2  (\ell-1)^2 + \sum_{i=2}^{\ell} (\ell- i+1)^2 (2\ell-i) \right)
%    	\\ &=
%    \]
 \end{proof}

\subsection{Asymptotic normality of  pattern counts} \label{S:distr_patt}

Here we apply Theorem~\ref{T:main} to derive the limiting distribution of the pattern counts $N_{H,\thetavec}(\Tran)$. 
In fact,  all applications of Theorem~\ref{T:main} typically have short proofs leaving the lower bound for the variance to be the most technically involved part. 

\begin{thm} \label{T:CLT_patterns}
Let $H$ be a tree on $\ell$ vertices and $\thetavec\in\{0,1\}^{\ell}$ be a non-zero vector. Then $N_{H,\thetavec}(\Tran)$ is asymptotically normal and $\delta_{\mathrm{K}}\left[N_{H,\thetavec}(\Tran)\right]=O(n^{-1/4 + \eps})$ for  any $\eps>0$.
\end{thm}

\begin{proof}
For a tree $T\in\mathcal{T}_n$, let $F(T)$ be the number of occurrences of $(H,\thetavec)$ in the induced subforest of $T$ 
for the set of vertices with degrees at most $\log n$ in $T$.

Removing one edge from $T$  can only destroy  at most  $ \log^{\ell} n$ patterns $(H,\thetavec)$ counted in $F(T)$.
Thus, $F$ is  $\alpha$-Lipschitz with $\alpha = 2 \log^{\ell} n$. If two perturbations $\swap_{i}^{jk}$ 
and $\swap_{a}^{bc}$ are at distance at least 
 $3\ell$ in $T$ then every pattern  $(H,\thetavec)$  counted  in $F(\swap_{i}^{jk} \swap_{a}^{bc} T) - F(T)$ 
 (with positive or negative sign) is present in exactly one of  the terms
 $F(\swap_{i}^{jk}T) - F(T)$ and $F(\swap_{a}^{bc} T) - F(T)$ (with the same sign). Thus, 
 $F$ is  $\rho$-superposable with $\rho= 3 \ell$.
 
From \eqref{eq:Moon}, we know that  
\[ 
 \Pr (F(\Tran)\neq N_{H,\thetavec}(\Tran)) =  e^{-\omega(\log n)}.
 \] 
 Since the values of these random variables are not bigger than $n^{\ell}$, we get 
  \begin{align*}
   \Exp {F(\Tran)} &=  \Exp {N_{H,\thetavec}(\Tran)} +  e^{-\omega(\log n)},\\
      \Vari {F(\Tran)} &=  \Vari {N_{H,\thetavec}(\Tran)} + e^{-\omega(\log n)}.
   \end{align*}
Combining \cref{L:Var} and \cref{l:Var_pattern}, we get that $\Vari {F(\Tran)}  = \Omega(n)$.
 Applying Theorem~\ref{T:main}, we complete the proof.
\end{proof}

In the next result we allow the pattern to grow, but restricted to the case when $H$ is a path and all $\theta_i$ equal $1$ (all vertices are ``empty'').
\begin{thm}\label{T:CLT_paths}
Let  $\ell = O(n^{1/8-\delta})$ for some fixed $\delta \in (0,1/8)$. Then $P_\ell(\Tran)$ is asymptotically normal 
and  $\delta_{\mathrm{K}}\left[P_{\ell}(\Tran)\right]=O(n^{-\varepsilon'})$ for any $\varepsilon' \in (0, 2\delta)$.
\end{thm}

\begin{proof}
%For $\ell$ not very large, an asymptotical normality of $P_{\ell}(\Tran)$ follows from Theorem~\ref{T:main}. Indeed, 
For a tree $T\in \Tset_n$, let 
\[
V_{\rm{good}}(T) := \left\{ i \in [n]  \st  \text{for all $d\in [n]$, we have } 
 |\{j \in[n] \st d_T(i,j) = d\}| \leq d \log^4 n\right\}.
\]
Define $F(T)$  to be the number of induced paths on $\ell$ vertices in  the forest $T[V_{\rm{good}}(T)]$.

The number of $\ell$-paths counted in $F(T)$ containing any fixed edge  is at most 
\[
	 \log^8 n \sum_{i = 1}^{\ell-2}  i (\ell-i-1) \leq \dfrac12\ell^3 \log^8 n.
\]
Arguing similarly to the proof of \cref{T:CLT_patterns}, we conclude that  $F$ is $\alpha$-Lipschitz 
with $\alpha= \ell^3 \log^8 n$ and $\rho$-superposable with $\rho=3\ell$.
From \cref{T:beta},  we also get 
\[ 
 \Pr (F(\Tran)\neq P_{\ell}(\Tran)) =  e^{-\omega(\log n)}.
 \] 
 Next, for a  tree $T\in \Tset_n$,  observe that $F(T)\leq P_{\ell}(T) \leq n^2$  since any path in $T$ 
 is uniquely determined by the choice of its end vertices. 
 The rest of the argument is identical to  the proof of \cref{T:CLT_patterns}.
\end{proof}

%%%%%%%%%%%%%%%%%%%%%%%%%%%%%%%%%%%%%%
%%%%%%%%%%%%%%%%%%%%%%%%%%%%%%%%%%%%%%
%%%%%%%%%%%%%%%%%%%%%%%%%%%%%%%%%%%%%%
%%%%%%%%%%%%%%%%%%%%%%%%%%%%%%%%%%%%%%
%%%%%%%%%%%%%%%%%%%%%%%%%%%%%%%%%%%%%%%%%%%%%%%%%%%%%%%%%%%%
%%%%%%%%%%%%%%%
\section{Number of automorphisms}\label{S:auto}
%%%%%%%%%%%%%%%%%%%%%%%%%%%%%%%%%%%%%%
%%%%%%%%%%%%%%%%%%%%%%%%%%%%%%%%%%%%%%
%%%%%%%%%%%%%%%%%%%%%%%%%%%%%%%%%%%%%%
%%%%%%%%%%%%%%%%%%%%%%%%%%%%%%%%%%%%%%
%%%%%%%%%%%%%%%%%%%%%%%%%%%%%%%%%%%%%%%%%%%%%%%%%%%%%%%%%%%%
%%%%%%%%%%%%%%%

An \emph{automorphism} of a graph $G$ is a bijection $\sigma: V(G) \to V(G)$ such that the edge set of $G$ is preserved under $\sigma$.  Bona and Flajolet \cite{Bona2009} studied this parameter for random unlabelled rooted non-plane trees and random phylogenetic trees (rooted non-plane binary trees with labelled leaves). They showed that in both cases  the distribution is \emph{asymptotically lognormal}; that is, the logarithm of the number of automorphisms in a random tree is asymptotically normal.  McKeon \cite{McKeon1996} proved asymptotic formulas for the number automorphisms in related random models of  unlabelled locally restricted trees. 

 In her PhD thesis Yu \cite{Yu2012} determined the asymptotics of $\Exp{\log \Autabs{\Tran}}$ for uniform random labelled tree  $\Tran$. She also made the following conjecture:

%A \emph{rooted tree automorphism} is an automorphism $\sigma$ of a tree $T$ with root $r$ such that $\sigma(r) = r$. $\Autr{T}$ is the group of rooted automorphisms of $T$. 

%\begin{definition}
%	A rooted tree automorphism is an automorphism $\sigma$ of a tree $T$ with root $r$ such that $\sigma(r) = r$. $\Autr{T}$ is the group of rooted automorphisms of $T$.
%\end{definition}

\begin{conj}\label{conj:log-normal}\cite{Yu2012}
	The distribution of $\Autabs{\Tran}$ is asymptotically lognormal.
\end{conj}

In this section we prove this conjecture. 
%The difference between the distribution of $\log\Autabs{T}$ and $\log\Autrabs{T}$ turns out to be small. We prove lognormality for rooted trees via Theorem \ref{T:main}; the unrooted result follows almost immediately due to a relation given later.
Unfortunately, we cannot immediately apply Theorem \ref{T:main} to derive the distribution of the number of automorphisms since the logarithm of
this parameter  is not $\rho$-superposable for a sufficiently small $\rho$. This happens because some trees have automorphisms 
 affected by both perturbations   $\swap_{i}^{jk}$ and $\swap_{a}^{bc}$ even if $d_T(\{j,k\}, \{b,c\})$ is large.
Instead,  we start by looking at $\Autr{T}$, the subgroup of $\Aut{T}$ consisting of automorphisms $\sigma \in \Aut{T}$
 such that $\sigma(r) = r$, where $r$ is some fixed vertex from $[n]$. In other words, $\Autr{T}$ is the set of root preserving automorphisms of a tree $T$ with root $r$, or equivalently the stabilizer of $r$.

 The parameter $|\Autr{\Tran}|$ is easier to work with while also remaining asymptotically very similar to $|\Aut{\Tran}|$. 
 The ease of analysis comes from the product representation of $\Autrabs{T}$ given by Yu \cite[Corollary 2.1.3]{Yu2012}. 
\begin{align}\label{eqn:yu-product}
	\Autrabs{T} &= \prod_{i \in [n]} \prod_{B}N_i(B,T,r)!
\end{align}
The product over $B$ represents a product over isomorphism classes of rooted unlabelled trees. Define a \emph{branch} of $T$ at $v$ to be a subtree rooted at an immediate descendent (with respect to $r$) of $v$. 
That is the branch is a fringe subtree of $T$ at this descendent. The term $N_i(B,T,r)$ denotes the number of branches isomorphic to $B$ at vertex $i$. Factorisation (\ref{eqn:yu-product}) also follows from the result of Stacey and Holton that says every rooted automorphism is a product of branch transpositions \cite[Lemma 2.4]{StaceyHolton1974}. 
%\begin{align}\label{eqn:yu-product}
%\Autrabs{T} &= \prod_{d,B} d!^{N(T,d,B)}
%\end{align}
%where $B$ represents a product over isomorphism classes of rooted unlabelled trees and $N(T,d,B)$ is the number of vertices with exactly $d$ branches isomorphic to $B$ in $T$. This can be rewritten as 

We give an example of \eqref{eqn:yu-product} in Figure \ref{fig:example-tree} for a tree on $9$ vertices. There are only three types of branches in this tree with repsect to the root $r=1$,  namely $B_1$, $B_2$, and $B_3$. Vertex 1 has two branches isomorphic to $B_2$, and thus $N_1(B_2,T,r)! = 2! = 2$. It also has one branch isomorphic to $B_1$, and thus $N_1(B_1,T,r)! = 1$. Vertex 2 has three branches isomorphic to $B_3$, and thus $N_2(B_3, T, r)! = 3! = 6$. Vertices 3 and 4 each have one branch isomorphic to $B_3$, and thus $N_3(B_3,T, r)! = N_4(B_3, T,r)! = 1$. 
Applying (\ref{eqn:yu-product}) shows that   $|\aut_r(T)| = 3! \cdot 2! = 12$.

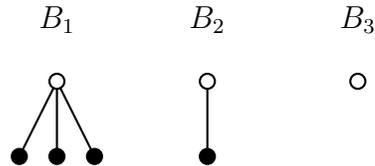
\begin{figure}[h!]
	\begin{center}
		\begin{tikzpicture}
		
		\tikzset{ vert/.style={black, fill=black, thick} }
		\tikzset{ vert2/.style={black, fill=white, thick} }
		\tikzset{ edge/.style={ thick, black } }
		
		% % % % tree
		%edges
		\draw[edge] (1,2) to (0,1);
		\draw[edge] (1,2) to (1,1);
		\draw[edge] (1,2) to (2,1);
		\draw[edge] (0,1) to (-0.5,0);
		\draw[edge] (0,1) to (0,0);
		\draw[edge] (0,1) to (0.5,0);
		\draw[edge] (1,1) to (1,0);
		\draw[edge] (2,1) to (2,0);
		%labels
%		\node[above] at (1,2.5) {$T$};
		\node[above] at (1,2) {$1$};
		\node[left] at (0,1) {$2$};
		\node[left] at (1,1) {$3$};
		\node[left] at (2,1) {$4$};
		\node[below] at (-0.5,0) {$5$};
		\node[below] at (0,0) {$6$};
		\node[below] at (0.5,0) {$7$};
		\node[below] at (1,0) {$8$};
		\node[below] at (2,0) {$9$};
		\node[above] at (1, 3) {Tree $T$ with root $r=1$};
		%vertices
		\draw[vert2] (1,2) circle [radius=0.1];
		\draw[vert] (0,1) circle [radius=0.1];
		\draw[vert] (1,1) circle [radius=0.1];
		\draw[vert] (2,1) circle [radius=0.1];
		\draw[vert] (-0.5,0) circle [radius=0.1];
		\draw[vert] (0,0) circle [radius=0.1];
		\draw[vert] (0.5,0) circle [radius=0.1];
		\draw[vert] (1,0) circle [radius=0.1];
		\draw[vert] (2,0) circle [radius=0.1];
		
		% % % % branches
		
		%branch1
		\begin{scope}[shift={(8,0)}]
		\node[above] at (0,1.5) {${B_2}$};
		\draw[edge] (0,0) to (0,1);
		\draw[vert] (0,0) circle [radius=0.1];
		\draw[vert2] (0,1) circle [radius=0.1];
		\node[above] at (0,3) {Branches of $T$};
		\end{scope}
		
		%branch2
		\begin{scope}[shift={(10,0)}]
		\node[above] at (0,1.5) {${B_3}$};
%		\draw[edge] (0,1) to (0.5,0);
%		\draw[edge] (0,1) to (-0.5,0);
		\draw[vert2] (0,1) circle [radius=0.1];
%		\draw[vert] (0.5,0) circle [radius=0.1];
%		\draw[vert] (-0.5,0) circle [radius=0.1];
		\end{scope}
		
		%branch3
		\begin{scope}[shift={(6,0)}]
		\node[above] at (0,1.5) {${B_1}$};
		\draw[edge] (0,0) to (0,1);
		\draw[edge] (0.5,0) to (0,1);
		\draw[edge] (-0.5,0) to (0,1);
		\draw[vert] (0,0) circle [radius=0.1];
		\draw[vert] (0.5,0) circle [radius=0.1];
		\draw[vert] (-0.5,0) circle [radius=0.1];
		\draw[vert2] (0,1) circle [radius=0.1];
		\end{scope}
		\end{tikzpicture}
	\end{center}
\caption{\label{fig:example-tree} A labelled tree on the left and its (rooted, unlabelled) branches on the right.}
\end{figure}

To define our tree parameter $F(T)$, we look at a subgroup of $\Autr{T}$ based on small automorphisms. We define a \emph{small branch} to be a branch with at most $4\log n$ vertices, any branch that is not small is \emph{large}. A \emph{small automorphism} is an automorphism where any vertex that is the root of a large branch is fixed. For a given tree $T$, let $\Autsmall \subseteq \Autr{T}$ be the set of small automorphisms.

\begin{lemma}
	$\Autsmall$ is a subgroup of $\Autr{T}$. 
\end{lemma}
	
\begin{proof}
	Observe that any automorphism in $\Autsmall$ must also have an inverse in $\Autsmall$, since they move the same vertices. Furthermore, to prove closure under composition, suppose that $a,b \in \Autsmall$ but $ab \notin \Autsmall$. Let $B$ be a large branch that is mapped by $ab$ onto $B^\prime$. Then all of the vertices in $B$ are moved by either $a$ or $b$. Since $a \in \Autsmall$, there are some vertices in $B$ not moved by $a$; denote this set by $X$. Since $B$ is connected, there exists an edge between $X$ and $V(B)\backslash X$ in the edge set of $B$. Thus there exists an edge between $aX$ and $aV(B)$ in $T$; however this creates a cycle and thus a contradiction. Thus $ab$ must also only move small branches, and thus $ab \in \Autsmall$. Thus $\Autsmall$ is a subgroup. 
\end{proof}

The parameter $F(T)$ is obtained by writing $|\Autsmall|$ in the same product representation as $\Autrabs{T}$ and taking the logarithm:
%\begin{align}\label{eqn:function-definition}
%F(T) &= \sum_{d = 1}^{4\log n}\sum_{|B| = 1}^{4\log n} N(T,d,B)\log d!\nonumber\\
%&= \sum_{i,B} \log (N_i(B)!) \Ind{|B| \leq 4\log n,\ N_i(B) \leq 4\log n}.
%\end{align}
\begin{align}\label{eqn:function-definition}
F(T) &:= \log |\Autsmall| = \sum_{i\in [n]}\sum_{B \in \Bsmall} \log (N_i(B, T,r)!).
\end{align}
Here $\Bsmall$ is the set of small branches.  

\remark{
     In fact, the parameter $F$ defined  above belongs to a larger class of  additive functionals 
     considered by Janson \cite{Janson2016} and Wagner \cite{Wagner2014}. They established a general CLT for this type of parameters.   \cite[Theorem 1.3]{Janson2016} and  \cite[Theorem 2]{Wagner2014} do not cover the number
     of automorphisms in $\Tran$ because 
     $
        \E \left[ \left(\sum_{B} \log (N_i(B, \Tran, r)!)\right)^{2} \right]
      $
      is not vanishing. In fact, it is bounded below by the second moment of the number of leaves attached to a given vertex which tends to a positive constant;  see also the estimates given in
      \cite[Appendix B]{TreesArxiv}. 
}

Next, we show that  $F(T)$  satisfies assumptions   of Theorem \ref{T:main} while also being very close to $\log \Autrabs{T}$.

\begin{lemma}\label{lem:F(T)-properties}
	Let $\alpha = 3\log n$ and $\rho = 10\log n$. Then $F(\Tran)$ as defined in (\ref{eqn:function-definition}) is $\alpha$-Lipschitz and $\rho$-superposable.
\end{lemma}

\begin{proof}
	To prove the Lipschitz property, we show that for any two trees $T$ and $T^\prime$ differing by a perturbation $\swap_i^{jk}$, the order of $\Autsmall$ for each tree can differ by at most a factor of $n^3$. Any automorphism of $T$ fixing $\left\{i,j,k \right\}$ is an automorphism of $T^\prime$, since all other edges remained static so their orbits are unaffected. Let $G_{ijk}$ be the subgroup of $\Autsmall$ that fixes $\left\{ i,j,k \right\}$. Then the cosets of this subgroup are defined by where they send each of these vertices. Since there are at most $n$ such options for each element in the set,  we get at most $n^3$ cosets. By Lagrange's theorem, we get that 
	\begin{align*}
		|\Autsmall(T)| \geq \frac{|\Autsmall(T^\prime)|}{n^3}
	\end{align*} 
	and vice versa by swapping the roles of $T$ and $T^\prime$. Taking the logarithm of both sides gives the desired bound. 
		
	Next, we show that $F$ is $\rho$-superposable. By the definition of this property and the triangle inequality, it is sufficient to prove that $\log(N_i(B,T,r)!)$ is $\rho$-superposible for all $i\in[n]$ and $B\in\Bsmall$. 
%	Combining the previous result with the triangle inequality gives the immediate upper bound of
%	\begin{align*}
%		\big| \AutF{T} -  \AutF{\swap_i^{jk}T} -  \AutF{\swap_a^{bc}T} + \AutF {\swap_i^{jk}\swap_a^{bc}T} \big| &\leq 	\big| \AutF{T} -  \AutF{\swap_i^{jk}T} \big|\\
%		&\quad + \big| \AutF{\swap_a^{bc}T} - \AutF {\swap_i^{jk}\swap_a^{bc}T} \big|\\
%		&\leq 6\log n.
%	\end{align*}
%	However, this bound is insufficient for larger $d$. Instead we argue that $\rho(d) = 0$ for all $d > 10\log n$. 
	Suppose 
	$$
	d = d_T \left( \left\{ j,k \right\}, \left\{ b,c\right\} \right) > 10\log n.
	$$
	Then suppose an $i$-branch isomorphic to $B$ %automorphism $\sigma \in \aut_{\rm small}(T)$
	 is created or destroyed by $\swap^{jk}_i$. %Then $\sigma$ must not fix $\left\{ i, j, k \right\}$.
	  Any path between one of $\left\{j,k\right\}$ and one of $\left\{b,c\right\}$ must be longer than $10\log n$. Therefore, any $i$-branch isomorphic to $B$ does not meet $\{a,b,c\}$. % any parent vertex in the tree is strictly more than $5\log n$ distance from at least one vertex in each pair. So $\sigma$ must fix $\left\{a,b,c\right\}$ and all lower branches, since each branch moved by the automorphism is at most $4\log n$.
	   So $\swap_a^{bc}$ cannot affect the presence or absence of an $i$-branch isomorphic to $B$. % $\sigma$ in $\Autsmall(T)$.	
%	Similarly, any automorphism created or destroyed by $\swap_a^{bc}$  can not be affected by $\swap_{i}^{jk}$.
	Thus,
	\[ 
 	F(\swap_{i}^{jk}\swap_{a}^{bc}T)	 - F(T) =
   \left(F(\swap_{i}^{jk}T) - F(T)\right) + 
   \left(F(\swap_{a}^{bc}T) - F(T)\right).
	\]
This completes the proof.
\end{proof}

In the next lemma we  derive bounds needed to compare  $|\aut(\Tran)|$ and $F(\Tran)$.
\begin{lemma}\label{lem:moment-relations}
	The following statements hold.
	\begin{enumerate}[label=(\alph*)]
		\item $\Big| \log \Autrabs{T} - \log \Autabs{T}\Big| \leq \log n$ for all $T\in \Tset_n$,
		\item $\Pro{F(\T) \neq \log\Autrabs{\T}} = O\left( \frac{1}{n^3} \right)$,
		\item $\Exp{\log \Autabs{\T}} - \Exp{F(\T)} = O(\log n)$,
		\item $\Vari{ \log \Autabs{\T}} - \Vari{F(\T)} = O\left(\sqrt{n \log^3 n}\right)$.
	\end{enumerate}
\end{lemma}

\begin{proof}
	Each automorphism in $\Autr{T}$ is an automorphism in $\Aut{T}$. 
	The group $\Aut{T}$ operates on $[n]$ such that $\Autr{T}$  is the stabilizer of $r$. Hence 
	\[
		 \Autrabs{T} \leq \Autabs{T}  = |\operatorname{Orbit}(r)| \times \Autrabs{T} \leq n \Autrabs{T}.
	\]
	Thus, we get (a).
	Parts (b) follows almost immediately from results by Yu \cite[Corollary~2.2.2%, Theorem 2.3.2, Lemma 6
	]{Yu2012}. To show part (c), we  use  parts (a) and (b)  and observe 
	$F(T) \leq \Autrabs{T} \leq \log n! \leq n \log n$ to get that
	\begin{align*}
		\Exp{\log \Autabs{\T} - F(\T)} &< \max_{T} \abs{ \log \Autabs{T} - \log \Autrabs{T} } \\
		& \quad +    \Pro{F(\T) \neq \log\Autrabs{\T}}  n\log n\\
		& \leq \log n + O\left( \frac{\log n}{n^2} \right)
		 = O\left( \log n \right).
	\end{align*}
	 Finally, we proceed to part  (d).
	  Let $W = F(\Tran) - \log\Autrabs{\Tran}$ and $Z = \log \Autrabs{\Tran} - \log\Autabs{\Tran}$.
	   From \cref{lem:moment-relations}(a,b,c), we get that 
	   \begin{align*}
	   	|\Vari{W} +  \Cov(Z,W)|  &\leq  \Pro{F(\T) \neq \log\Autrabs{\T}}  2 n^2\log^2 n = O\left(\frac{\log^2 n} {n}\right)
	   	\\
	   		   	\Vari{Z}  &\leq \E Z^2 \leq  \log^2 n, 
	   		   	\\
	   		| \Cova{F(\Tran),W+Z}| &\leq \left( \Vari{F(\Tran)}  \Vari{W+Z}\right)^{1/2}
	   		 = O(\sqrt{n \log^3 n}).
	   \end{align*}
	   Then, we have 
%	\begin{align*}
%		\Vari{\log \Autrabs{\T}} &= \Vari{ F(\T) + \log \Autrabs{\T} - F(\T)}\\
%				&= \Vari{F(\T)} + \Vari{\log \Autrabs{\T} - F(\T)} + 2\Cova{F(\T), \log \Autrabs{\T} - F(\T)}\\
%				&= \Vari{F(\T)} + O\left( \frac{\log^2n}{n} \right).
%	\end{align*} 
%	Finally we relate this back to unrooted automorphisms. Let $W = F(\Tran) - \log\Autrabs{\Tran}$ and $Z = \log \Autrabs{\Tran} - \log\Autabs{\Tran}$. Then 
	\begin{align*}
		\Vari{\log \Autabs{\T}} = \Vari{F(\T) + W + Z}= \Vari{F(\T)} +  O(\sqrt{n \log^3 n}).
	\end{align*}
\end{proof}

%To show that $F(T)$ and $\log \Autrabs{T}$ have the same asymptotic distribution, we use the following two lemmas by Yu \cite{Yu2012}. In particular, Lemma \ref{lem:yu-no-big-autos} gives some intuition to why $F(T)$ is not far from $\log \Autrabs{T}$ asymptotically.
%
%\begin{lemma}(\cite{Yu2012})\label{lem:yu-no-big-autos}
%	Let $N_{m,v}$ be the number of ordered pairs of isomorphic $m$-vertex branches adjacent to $v$. Let 
%	\begin{align*}
%	N = \sum_{v=1}^{n} \sum_{m \geq \xi} N_{m,v}
%	\end{align*}
%	where $\xi = (a + 1) \log n$ for $a \geq 1$ is some constant. Then $\Pro{N \neq 0} = O\left( \frac{1}{n^a} \right)$.
%\end{lemma}

%\begin{lemma}(\cite{Yu2012}) \label{lem:aut-expectation}
%	$\Exp{\log \Autabs{T}} = \mu n + O(\log^2 n)$, where $\mu \approx 0.051$.
%\end{lemma}

The final ingredient needed to apply Theorem \ref{T:main} is a bound on the variance of $F(\Tran)$, given in the lemma below. 
%The proof of \cref{lem:linear-moments} is long and technical and is thus postponed until \cref{sec:auto-variance}.
\begin{lemma}\label{lem:linear-moments}
	For sufficiently large $n$, we have $\Vari{F(\T)} \geq 0.002\,n$.
\end{lemma}
The proof of Lemma \ref{lem:linear-moments} is lengthy and quite technical. We include it 
 in Appendix B of the ArXiv version  \cite{TreesArxiv}  of  the current paper.

Now, we are ready to prove the following result.

\begin{thm}\label{thm:log-normal}
	Conjecture \ref{conj:log-normal} is true. Furthermore, $\delta_K\left[ \log \Autabs{\T} \right] = O\left(n^{-\frac{1}{4} + \epsilon}\right)$ and $\delta_K\left[ \log \Autrabs{\T} \right]= O\left(n^{-\frac{1}{4} + \epsilon}\right)$ for any $\epsilon > 0$.
\end{thm}

\begin{proof}[Proof of Theorem \ref{thm:log-normal}]
		Combining     \cref{lem:F(T)-properties},  \cref{lem:moment-relations}, and \cref{lem:linear-moments}, we get that 
		the parameter $F$ defined in \eqref{eqn:function-definition} satisfies all the assumptions of Theorem \ref{T:main} 
		and $\delta_k[F(\Tran)] = O(n^{-1/4 + \epsilon})$  for any $\epsilon >0$.
	Using  \cref{lem:linear-moments} and recalling that $F(T) \leq \log\Autrabs{\T} \leq \log|\Aut{\T}|$, 
	we get $\log\Autrabs{\T}$ and $\log|\Aut{\T}|$  has the same limiting distribution (with the same bound for the Kolmogorov distance).
\end{proof}

\begin{remark}
	%A recent result by Wagner \cite{Wagner-Talk} gives the variance, \as{as well as the distribution,} of $\log\Autrabs{T_n}$ (and thus, with minimal work, $\Vari{F(T_n)}$) with far greater accuracy than our lemma (specifically that $\sigma^2(\log \Autabs{T_n}) \approx 0.039498n + O(1)$). However, our result is sufficient to apply Theorem \ref{T:main} and prove the conjecture. Part (a) of Lemma \ref{lem:F(T)-properties} argues that automorphisms of ``large'' branches almost surely do not contribute to the moments of $\log\Autabs{T_n}$.

	Recently, Stufler and Wagner \cite{Wagner-talk}  have also announced progress in showing that the distributions of $\Autabs{T}$ and $\Autrabs{T}$ are asymptotically lognormal; however, 
	 it has not yet appear in any published or arXiv paper. Their method is based on the analysis of the  generating function and is  different from our approach.  Stufler and Wagner  gave  much more accurate values for the mean and variance in their talk  \cite{Wagner-talk}, specifically $\Exp{\log\Autabs{T}} \approx 0.052290n$ and $\Vari{\log\Autabs{T}} = 0.039498n$.
\end{remark}

\section{Tools from the theory of martingales}\label{S:m_theory}
%In this section, we recall the definition of a martingale  and state general results about its limiting distribution and concentration properties. 

 Let   $\mathcal{P}=(\varOmega,\calF, \Pr)$ be a {probability} space.
A sequence $\calF_0,\ldots,\calF_n$ of {sub-$\sigma$-fields}
of $\calF$ is a  {\textit{filtration}} if $\calF_0\subseteq\cdots\subseteq\calF_n$.
%If $X$ is a random variable on  $\mathcal{P}$, we use the following 
%notation $\E_j X = \E (X \st \calF_j)$.
A sequence $Y_0,\ldots,Y_n$ of random variables on  $\mathcal{P}$
is a \textit{martingale with respect to $\calF_0,\ldots,\calF_n$} if
\begin{itemize}\itemsep=0pt
\item[(i)] $Y_i$ is $\calF_i$-measurable and {$\left| Y_i \right|$} has finite expectation, for $0\le i\le n$;
\item[(ii)] $\Exp{ Y_i  \mid \calF_i} = Y_{i-1}$ for $1\le i\le n$.
\end{itemize}
In the  following we will always assume that $\calF_0 = \{\emptyset, \varOmega \}$ and so 
$Y_0 = \E [ Y_n]$.

In this section we state some general results  on concentration and limiting distribution for martingales. 
In fact, we only need these results for discrete uniform  probability spaces, where  the concept of martingale
reduces to  average values over increasing set systems. 
In this case, $\varOmega$ is a finite set and
each $\sigma$-field $\calF_{i}$ is  generated by unions of blocks of a partion of $\varOmega$.
Following McDiarmid  \cite{McDiarmid}, for $i=0,\ldots, n$ 
we define the \emph{conditional range} of a random variable $X$ on $\mathcal{P}$ as
\begin{equation}\label{con-range} 
  \ran [X \mid \calF_{i}] := 
  \sup [X\mid \calF_{i}] +
  \sup[- X\mid \calF_{i}].
\end{equation}
Here,  $\sup[X\mid \calF_{i}] $ is the $\calF_{i}$-measurable random variable which  takes the 
value  at $\omega \in \varOmega$ equal to the maximum value of $X$ over the block of $\calF_{i}$ containing $\omega$ (and similarly for $-X$).  More generally,  ``supremum'' can be replaced by ``essential supremum''.   For more information about conditional range and diameter,  see, for example, \cite[Section 2.1]{mother} and references therein. 
 We will use that 
the conditional range is a seminorm and, in particular, it is subadditive. 
%When $\calF_0 = \{\emptyset, \varOmega \}$ we also write $\ran  X = \ran_0 X$.

Our first tool is the following result of McDiarmid \cite{McDiarmid}. Further in this section, the notation 
$\ran_i [\cdot]$ stands for $\ran [\cdot \mid \calF_{i}]$.
\begin{thm} \emph{(\cite[Theorem 3.14]{McDiarmid})} \label{T:concentration}
%Suppose that $\mathcal{P}=(\varOmega,\calF,\mathbb{P})$ is a finite probability space. 
Let $Y_0,Y_1,\ldots, Y_n$ be a real-valued martingale  with respect to the filtration $\{\emptyset,\varOmega\}=\calF_0,\calF_1\ldots,\calF_n$. 
Denote  
\[
	R^2 :=   \sum_{i=1}^n \left( \operatorname{ran}_{i-1} [Y_i] \right)^2. 
\]	
 %Assume
%\[\sum_{i=1}^n \left( \operatorname{ran}_{i-1}Y_i \right)^2
 %\leq \hat{r}^2 
%\]
%for some $\hat{r}>0$.
 Then, for any $r,t > 0$  %\ac{What is $\hat{r}$ compared to just $r$ here? Should they just be $r$?},
\[ \Pr\left(|Y_n - Y_0| \geq t \right) 
\leq 2\exp( -2t^2/r^2) + 2 \Pr \left( R^2 > r^2 \right).
\]
\end{thm}

The normalized quadratic variation of  a martingale sequence $\Y= (Y_0, \ldots, Y_n)$ is defined by
\[
	Q[\Y] :=  \frac{1}{\Vari{Y_n}} \sum_{i=1}^n (Y_i - Y_{i-1})^2. 
\]
Observe that 
\begin{equation}\label{telescope}
	\Exp{ (Y_i - Y_{i-1})^2}  = \Exp{\Var[Y_i \mid F_{i-1}]}=
	\Exp{\E \left[ Y_i^2  - Y_{i-1}^2  \mid F_{i-1}\right]  }=
	\Exp{Y_i^2 -  Y_{i-1}^2}.
\end{equation}
Thus, 
\[
\E Q[\Y] =     \frac{1}{\Vari{Y_n}}  \sum_{i=1}^n  \left(\Exp{Y_i^2} - \Exp{Y_{i-1}^2}\right)= 1.
\] 
A classical result by Brown \cite{Brown1971} states that if the  
increments $Y_{i}-Y_{i-1}$ have finite variances,
$
	Q[\Y] \xrightarrow{prob.} 1
$
 as $n \rightarrow \infty$
and a certain Lindeberg-type condition is satisfied then  the limiting distribution of $Y_n$  is  normal, i.e. 
$\delta_{\rm K}[Y_n] \rightarrow 0$.
 For a more restricted class of martingales with bounded  differences
 these conditions can be slightly simplified and will be sufficient for our purposes. Our second tool 
 is  the following result of Mourrat \cite{Mourrat2013}  which gives an 
  explicit  bound on the rate of convergence in the CLT under a 
 strengthened condition that  the normalized quadratic variation $Q[\Y]$ converges  to $1$  in $L^p$.
\begin{thm}
\emph{(\cite[Theorem 1.5.]{Mourrat2013})} \label{T:distr}
	Let $p \in [1,+\infty)$ and $\gamma \in (0,+\infty)$. There exists a constant 
	$C_{p,\gamma}>0$ such that, for any real 
	martingale sequence $\Y= (Y_0, \ldots, Y_n)$ 
	satisfying $|Y_i - Y_{i-1}| \leq \gamma$ for all $i=1,\ldots, n$,
	\[
	 \delta_{\rm K} [Y_n]
	 \leq  C_{p,\gamma} \left(
	  \frac{n \log n}{(\Vari{Y_n})^{3/2}} +   \left( \Exp{|Q[\Y] -1|^p}  + (\Vari{ Y_n})^{-p}\right)^{1/(2p+1)}
	 \right).
	\]
\end{thm}
One way to bound the  term  $\Exp{ |Q[\Y] -1|^p}$  in the above is by applying 
Theorem \ref{T:concentration} to the martingale for $Q[\Y]$ with respect to the same filtration, as{which gives the following lemma}.
\begin{lemma}\label{l:CLTbounds}
Let $Y_0,\ldots, Y_n$ be a real-valued martingale  with respect to the filtration $\{\emptyset,\varOmega\}=\calF_0,\ldots,\calF_n$. 
For $\hat{q}>0$, let $\calA_{\hat{q}}$ denote the event 
 \[
  \sum_{i=1}^n  
     \left(\operatorname{ran}_{i-1}  \left[\Vari{ Y_n  \mid \calF_{i}}\right] + (\ran_{i-1} [Y_i] )^2
  \right)^2
   > \left( \hat{q}\,\Vari{Y_n}\right)^2.
 \]
 Then, for any  $p \in [1,+\infty)$, we have 
\[ 
 \Exp{|Q[\Y] -1|^p} \leq  c_p \,\hat{q}^{p} 
 + 2 \Pr \left(\calA_{\hat{q}}\right)   \sup |Q[\Y] -1|^p,
\]
where $  c_p = 2p \int_{0}^{+\infty} e^{-2x^2} x^{p-1} dx$.
\end{lemma} 
\begin{proof}
  By definition, we have that $|Y_i - Y_{i-1}| \leq \ran_{i-1} [Y_i]$ for all $i \in [n]$. Therefore,
    \[
    	 \ran_{i-1} \left[(Y_i - Y_{i-1})^2\right] \leq (\ran_{i-1} [Y_i])^2.
    \]
Observe also $\ran_{i-1}  \left[ (Y_j - Y_{j-1})^2\right] = 0$ for any $j<i$.	 Then, using \eqref{telescope} and  the subadditivity of the conditional range, we get that  
	\begin{align*}
		\ran_{i-1} \left[ \Exp{Q[\Y]  \mid \calF_i } \right] & = 
		\frac{1}{\Vari{Y_n}}
		\ran_{i-1}
		\left[ \sum_{j=i}^{n} \Exp {(Y_{j}- Y_{j-1})^2 \mid \calF_i }\right]\\
		&=  \frac{\ran_{i-1}
		\left[ \Vari {Y_n \mid \calF_i} + (Y_i - Y_{i-1})^2 
		\right] }{\Vari {Y_n}}
				\\
		&
		 \leq 	  \frac{ \ran_{i-1} \left[\Vari{Y_n \mid \calF_i} \right ]+ \left(\ran_{i-1}[Y_i]\right)^2 }{\Vari {Y_n}} .
	\end{align*}
	Applying  Theorem \ref{T:concentration} to the martingale  
	$\{\Exp{Q[\Y]  \mid \calF_i }\}_{i=0,\ldots, n}$, we find that
	\[
		\Pr \left(|Q[\Y] -1| \geq t\right) \leq 2 \exp (-2t^2 /\hat{q}^2) + 2  \Pr \left(\calA_{\hat{q}}\right).
	\]
      Substituting this bound into 
      \[
      	     \Exp {|Q(\Y) -1|^p} = \int_{0}^{ t_{\max}}  \Pr \left(|Q(\Y) -1| \geq t\right) p t^{p-1}dt
      \]
      and changing the variable $t = \hat{q} x$, we complete the proof.  Here, $t_{\max}= \sup |Q(\Y) -1|$.
\end{proof}

Using the formulas for $\Exp{ (Y_j-Y_{j-1})^2  \mid \calF_{i}} $ similar to  \eqref{telescope}, we find that
\begin{equation}\label{telescope2}
	\Vari{ Y_n  \mid \calF_{i}}   = \sum_{j=i+1}^n \Exp { (Y_j-Y_{j-1})^2  \mid \calF_{i}}.
\end{equation}
Then, by the subadditivity of the conditional range, we get 
the next bound, which 
 will be useful in applying Lemma \ref{l:CLTbounds}.
\begin{equation}\label{eq:CLTbounds}
	\operatorname{ran}_{i-1}  \left[\Vari{ Y_n  \mid \calF_{i})}\right] 
	\leq \sum_{j=i+1}^n \ran_{i-1} \Exp{(Y_j - Y_{j-1})^2 \mid \calF_i}.
\end{equation}

The Doob martingale construction is another important tool in our argument. Suppose $\X = (X_1,\ldots,X_n)$
is a random vector on $\mathcal{P}$ taking values  in $S$
 and $f:S\rightarrow \Reals$ is such that $f(\X)$ has bounded expectation. Consider the filtration
 $\calF_0,
\ldots \calF_n$ defined by $\calF_i = \sigma(X_1,\ldots,X_i)$ which is the $\sigma$-field generated by random variables 
$X_1,\ldots X_i$. Then, the Doob martingale $\Y^{\rm Doob} = \Y^{\rm Doob} (f,\X)$ is defined by, for all $i=0,\ldots,n$,  
\begin{equation*}%\label{def_Doob}
	Y_i^{\rm Doob}  := \Exp{f(X_1,\ldots,X_n) \mid  \calF_i}.
\end{equation*}
In case  of finite $S$,  the random variables
$Y_i^{\rm Doob} $, 
$\Vari{Y_n^{\rm Doob}  \mid  \calF_j} $
 and $\ran_{i} [Y_n^{\rm Doob}] $ can be seen as functions 
 $f_i, v_i, r_i : S \rightarrow \Reals$
 of the random vector $\X$ defined as follows: for  $\xvec \in S$,  
\begin{equation}\label{var-as-func}
\begin{aligned}
	f_i(\xvec) &:= \Exp {f(\X) \mid X_1 = x_1, \ldots, X_i=x_i} =\Exp {f(x_1,\ldots,x_i, X_{i+1} \ldots, X_n)},\\
	v_i(\xvec) &:= \Vari {f(\X) \mid X_1 = x_1, \ldots, X_i=x_i}=  \Vari {f(x_1,\ldots,x_i, X_{i+1} \ldots, X_n)},\\
	r_{i}(\xvec) &:= \ran\left[f(\X) \mid X_1 = x_1, \ldots, X_i=x_i\right] 
	\\
	 &\phantom{:}=\max_{\yvec} f(x_1,\ldots,x_i, y_{i+1} \ldots, y_n) 
	     -\min_{\yvec}  f(x_1,\ldots,x_i, y_{i+1} \ldots, y_n),
\end{aligned}
\end{equation}
where $x_1,\ldots,x_i$ are fixed and $X_{i+1},\ldots, X_n$ are random and both $\max$ and $\min$ are 
over $\yvec \in S$ such that $y_j=x_j$ for $j=1,\ldots,i$.
If, in addition, random variables $X_1, \ldots, X_n$ are independent then
		\begin{equation}\label{eq:bound_range}
			|Y_i^{\rm Doob}  - Y_{i-1}^{\rm Doob} | \leq \ran_{i-1} \left[Y_i^{\rm Doob}\right]   \leq \max_{\xvec,\xvec'} |f(\xvec) - f(\xvec')|,
		\end{equation}
		where the maximum is over $\xvec, \xvec' \in S$ that differ only in the $i$-th coordinate.

In particular, the Doob martingale process is  applicable for functions of random permutations since we can represent them as vectors. 
Let $S_n$ be the set of permutations of $[n]$. We write $\omega =(\omega_1,\ldots,\omega_n) \in S_n$ 
if $\omega$ maps $j$ to $\omega_j$. The product of two permutuations  $\omega,\sigma \in S_n$ is defined by
\[
	\omega \circ \sigma := (\omega_{\sigma_1}, \ldots, \omega_{\sigma_n})
\]
which corresponds to {the composition of $\omega$ and $\sigma$} if we treat them as functions on $[n]$.  
% Let $\X = (X_1,\ldots,X_n)$  be a uniform random element of the set $S_n$ of permutations of $[n]$. 
% For a given permutation $\omega = (\omega_1,\ldots,\omega_n)$  we have
  %   \begin{equation*}
    %    f_k(\omega) = \E(f(\X) \mid X_j = \omega_j,\ 1 \leq j \leq k).
    % \end{equation*}
  % Note that $f_n(\omega) = f_{n-1}(\omega)=f(\omega)$ since first $n-1$ components 
   % determine uniquely $\omega$.    
     % Thus, the sequence 
   % \begin{equation}\label{Doob_perm}
     % \E f(\X) = Y_0(\X), Y_1(\X), \ldots, Y_{n-1}(\X)=f(\X)
   % \end{equation}
% forms a  martingale for $f(\X)$ with respect to the filter $\calF_0,\ldots,\calF_{n-1}$, where
%for each $k$, the $\sigma$-field $\calF_k$ is generated by sets 
%\[
%	\varOmega_{k,\nu} = \{\omega \in S_n \st \omega_j = \nu_j, \ 1\leq j\leq k\}
%\]
%  for all $k$-tuples $(\nu_1,\ldots,\nu_k)$   with distinct components.
%
%     
   For a function $f: S_n \rightarrow \Reals$ and $1\leq i\neq j \leq n-1$,  define 
  \begin{align*}
    \alpha_i [f] &:=  \sum_{a=i+1}^n \frac{\max_{\omega \in S_n} |f(\omega) - f(\omega \circ (ia))|}{n-i}, \\
        \varDelta_{ij} [f] &:=  \sum_{a=i+1}^n \sum_{b=j+1}^n \frac{
        \max_{\omega \in S_n} |f(\omega) - f(\omega \circ (ia)) - f(\omega \circ (jb)) + f(\omega \circ (jb)\circ(ia))|}{(n-i)(n-j)}.
  \end{align*}
   Let $\X = (X_1,\ldots,X_n)$  be a uniform random element of $S_n$  and $\Y^{\rm Doob}(f, \X)$ be the Doob martingale sequence for $f(\X)$. Note that  $Y^{\rm Doob}_n = Y^{\rm Doob}_{n-1} = f(\X)$ since the first $n-1$ coordinates $X_i$ determine the permutation  $\X$ uniquely.
   \begin{lemma}\label{Lemma_perm} If   $Y^{\rm Doob} = Y^{\rm Doob}(f,\X)$ where 
   $f: S_n \rightarrow \Reals$ and $\X$ is a uniform random element of $S_n$,  then
		\begin{itemize}
		\item[(a)]		$\displaystyle |Y^{\rm Doob}_i - Y^{\rm Doob}_{i-1}|\leq 	\ran_{i-1} 
		\left[ Y^{\rm Doob}_i \right] \leq \alpha_i[f],$ for  all $1\leq i \leq n-1$.
  \item[(b)]   $ \displaystyle	\ran_{i-1} \left[ \Exp{(Y^{\rm Doob}_j-Y^{\rm Doob}_{j-1})^2  \mid \calF_i} \right]\leq 2 \alpha_j[f] \varDelta_{ij}[f]$, 
    for all $1\leq i< j \leq n-1$.
		\end{itemize}
   \end{lemma}
   \begin{proof}
                   To show the first inequality in part (a), we observe that 
     \[ 
     -\mysup(-Y^{\rm Doob}_i \mid \calF_{i-1}) \leq Y^{\rm Doob}_{i-1} \leq \mysup (Y^{\rm Doob}_i \mid \calF_{i-1}),\]
    by definition.  The other bounds is a special case of \cite[Lemma 2.1.]{sister}
      for real-valued random variables,      where  the conditional range is the same as the
conditional diameter.
   \end{proof}

\section{Martingales for  tree parameters}\label{S:Construction}

To prove \cref{T:main} we use the martingale  based on the Aldous-Broder algorithm, which generates a random spanning tree of a given graph $G$.  Here is a quick summary: (1) consider the random walk starting from any vertex; (2) every time we traverse an edge which takes us to a vertex we have not yet explored, add this edge to the tree; (3) stop when we visited all  vertices.  The resulting random graph has uniform distribution  over the set of spanning trees of $G$,  for more details see \cite{Aldous1990}. If  $G$ is the complete graph $K_n$, $n\geq 2$, this construction can be rephrased as the following two-stage procedure  \cite[Algorithm 2]{Aldous1990}: 
\begin{itemize}
	\item[I.] For $1\leq i \leq n-1$ connect vertex $i+1$ to vertex $V_{i} = \min\{i, U_{i}\}$, 
	where $\U = (U_1,\ldots, U_{n-1})$ is uniformly distributed on $[n]^{n-1}$.
	\item[II.] Relabel vertices $1,\ldots,n$ as $X_1,\ldots,X_n$, where $\X = (X_1,\ldots,X_n)$
	 is a uniform random permutation from $S_n$.
\end{itemize}
 Let  $T(\uvec)$ is the tree produced  at stage I given that  $\U= \uvec$.  
 For a permutation 
 $\omega\in S_n$ and a tree $T\in \Tset_n$,  let
  \begin{equation*}
 \text{$T^{\omega}:=$  the tree obtained from $T$ by relabelling according
 to  $\omega$.}
 \end{equation*}
  From \cite{Aldous1990} we know that $T(\U)^{\X}$ has uniform distribution on the set $\Tset_n$.
  Now,  a tree parameter $F: \Tset_n \rightarrow \Reals $ can be seen as a function with domain $[n]^{n-1} \times S_n$. 
  % Essentially, the random vector $\U$ determines the structure of the tree. 
Consider the functions $\hat{F}: \Tset_n \rightarrow \Reals $ and
 $F_T: S_n \rightarrow \Reals$   defined by
 \begin{equation}\label{def_FF}
 	 \hat{F}(T) :=  \Exp{F(T^{\X})}, \qquad F_T(\omega) :=  F(T^{\omega}). 
 \end{equation}
Let  $\Y = (Y_0, \ldots, Y_{n-1})$ and  $\Z(T) = (Z_0(T), \ldots, Z_{n-1}(T))$ be  the Doob martingale sequences  
for $ \hat{F}(T(\U))$ and  $F_T(\X)$, respectively: for $i = 0,\ldots ,n-1$,
 \begin{equation}\label{def_YZ}
 	Y_i := \E \big[ \hat{F}(T(\U)) \mid \calF_i \big]
\qquad \text{and} \qquad
Z_i(T) :=  \E \big[ F_T(\X) \mid \calG_i \big],
 \end{equation}
 where   the filtrations are $\calF_i = \sigma(U_1,\ldots,U_i)$ and $\calG_i = \sigma(X_1,\ldots, X_{i})$. 
 We construct  the martingale 
 for $F(\Tran)$  by  combining the above  two sequences together. 
   Further in this section, we will use the following notations for conditional statistics of a random variable $W$ 
 with respect to  $\calF_i$ and $\calG_i$:
 \begin{equation*}
\begin{split}
  \E_{\calF_i} [W] &:=  \Exp{W \mid \calF_i},\\
 	\Var_{\calF_i} [W] &:=  \Vari{W \mid \calF_i}, \\
 	 	\mysup_{\calF_i} [W] &:=  \mysup [W \mid \calF_i],\\
 	 	 	\ran_{\calF_i} [W] &:=  \ran [W \mid \calF_i],
 \end{split}
 \qquad
 \begin{split}
 	\E_{\calG_i} [W] &:=  \E [W \mid \calG_i],\\
 	 \Var_{\calG_i} [W] &:= \Var [W \mid \calG_i],\\
 	 \mysup_{\calG_i} [W] &:=  \mysup [W \mid \calG_i],\\
 	 	 \ran_{\calG_i} [W] &:=  \ran [W \mid \calG_i].
\end{split}
 \end{equation*}
%Combining the above  sequences together, we obtain the martingale sequence  
% $(\Y, \Z(T(\U)))$ for $F(\Tran) = f(\U,\X)$ 
%\begin{equation}\label{big_mart}
%	\E f(\U,\X) = Y_0,  \ldots, Y_{n-1} = Z_0(T(\U)),  \ldots, Z_{n-1}(T(\U))
% =  f(\U, \X)
%\end{equation}
%with respect to the filter
%\[
%  \{\emptyset, \varOmega\}= \calF_0 \subseteq \cdots \subseteq \calF_{n-1} = 
%\calF_{n-1}\times \calG_0  \subseteq \cdots \subseteq \calF_{n-1}\times\calG_{n-1}.
%\] 

 \subsection{Properties of $F_T$  and  $\hat{F}$} \label{S:FF}
First,  we study properties of functions   $F_T$  and $\hat{F}$  from \eqref{def_FF} 
given that  the parameter $F$ is $\alpha$-Lipschitz and $\rho$-superposable.
   \begin{lemma}\label{l:FF}
    Let  a tree parameter $F: \Tset_n\rightarrow \Reals$ be $\alpha$-Lipschitz and  $\rho$-superposable
   for some $\alpha\geq 0$ and $\rho\geq 1$, then
       \begin{itemize}
       			\item[(a)]  $\hat{F}$ is $\alpha$-Lipschitz and $\rho$-superposable.
       			\end{itemize}
              	Furthermore, the following holds for all  trees $T \in \Tset_n$ and permutations $\omega \in S_n$.
              	\begin{itemize}
       			\item[(b)]   		
       			If $(ia)$ is  a transposition from $S_n$, then
       			 	\[
       			    |F_T(\omega) - F_T(\omega \circ (ia))| \leq  \alpha (\deg_T (i)+\deg_T (a)),
       			    \]
       			    where $\deg_T (i)$,  $\deg_T (a)$ are degrees of $i,a$ in the tree $T$.
       			 \item[(c)]
       			    Let $T' = \swap_{q}^{rs} T$ be a tree for some triple $(q,r,s)$.
       			    If    $(ia)$ is  a transposition from $S_n$ that     $d_T(\{i,a\}, \{r,s\}) \geq  \rho+1$, then
       			    \[
       			    	 F_T(\omega) - F_T(\omega \circ (ia) ) -  F_{T'}(\omega)  + 
       			    	 F_{T'}
       			    	 (\omega \circ (ia)) = 0.
       			    \]
       			   \item[(d)] 
       			   If $(ia), (jb)$ are transpositions from $S_n$ such that		    $d_T(\{i,a\}, \{j,b\}) \geq  \rho+2$, then
       			    \[ 
       			      F_T(\omega) - F_T(\omega \circ (ia) ) -  F_T(\omega \circ (jb) )  + F_T(\omega
       			   \circ (jb) \circ (ia)) = 0.
       			   \]

       \end{itemize}
 \end{lemma}
  \begin{proof}
  For any permutation  $\omega = (\omega_1,\ldots,\omega_n)\in S_n$ define the  function $F_\omega: \Tset_n \rightarrow \Reals$ by 
  $F_\omega(T):= F(T^\omega)$.     If $\swap_{i}^{jk}T$ is a tree then     $(\swap_{i}^{jk}T)^{\omega} =  \swap_{\omega_i}^{\omega_j \omega_k} T^\omega$. Relabelling also does not change the distances, that is, $d_T(a,b) = d_{T^{\omega}} (\omega_a,\omega_b)$ for all $a,b\in [n]$.  Thus, $F_\omega$ is $\alpha$-Lipschitz and $\rho$-superposable.  Averaging over all $\omega$ proves part (a).
  
  For part (b), we show that  the tree $T^{(ia)}$ can be obtained from $T$ by 
  performing at most $\deg_T(i)+\deg_T(a)$ tree perturbations $\swap_{x}^{yz}$. 
  We denote the set of these perturbations by  $\mathcal{P}^{ia}_T$.
%  Let $V_T(i)$ denote the set of vertices  adjacent to $i$ in $T$.  
  Let $u$ and $v$ be the vertices on the path from $i$  to $a$ in $T$  adjacent to $i$ and $a$, respectively.
  Consider  $\deg_T(i)-1$ perturbations  $\swap_{x}^{ia}$ for all vertices $x \neq u$ 
  adjacent to $i$   and $\deg_T(a)-1$ perturbations   $\swap_{x}^{ai}$ 
  for all  for all vertices $x \neq v$ 
  adjacent to $a$.    If $d_T(a,i) \leq 2$ then performing these 
  $\deg_T(i) + \deg_T(a)-2$ perturbations in any order turns $T$ into  $T^{(ia)}$. 
  Otherwise, all vertices $i,a,u,v$ are distinct and we need two more perturbations   $\swap_{i}^{uv}$  and  
  $\swap_{a}^{vu}$ to obtain  $T^{(ia)}$. 
  This defines the set $\mathcal{P}^{ia}_T$. Now,  since $F$ is $\alpha$-Lipschitz,
   the value of the function changes by at most $\alpha$ after each perturbation so
   \[
   	|F(T) - F(T^{(ia)})| \leq \alpha (\deg_T(i)+\deg_T(a)).
   \]
   The above holds for any $T\in \Tset_n$. 
   Substituting $T^{\omega}$ 
   and observing $\deg_{T^{\omega}}(\omega_i) =\deg_T(i)$, we prove part (b). 
   
   Before proving parts (c) and (d), we outline some important properties of the set  $\mathcal{P}^{ia}_T $ of the tree  perturbations that turn    $T$ into $T^{(ia)}$. 
      \begin{itemize}
     \item[(i)] 
       The   perturbations of $\mathcal{P}^{ia}_T$ can be performed in any order, that is, all intermediate graphs are trees.
        
              \item[(ii)] $d_T(\{x,y,z\}, \{i,a\}) =0$ for any $\swap_{x}^{yz} \in \mathcal{P}^{ia}_T$, 
                that is,  $\{x,y,z\} \cap \{i,a\} \neq \emptyset$ .  
\item[(iii)] 
 the distance from any $w\in[n]$ to $\{i,a\}$  is unchanged by perturbations  $\swap_{x}^{ia}$ or $\swap_{x}^{ai}$. 
 
 \item[(iv)] 
  the distance from any $w\in[n]$ to $\{i,a\}$ can  increase after performing 
  one of the perturbations $\swap_{i}^{uv}$ or $\swap_{a}^{vu}$ but then it decreases back to the initial value after performing 
  the second (so it never gets smaller than the initial distance $d_T(w,\{i,a\})$).
 \end{itemize}

For (c),  observe first that  $d_T(\{i,a\},\{r,s\}) \geq 2$ implies 
that  $i$ and $a$ are adjacent to the same sets of vertices in  $T$ and $T'$.
Consider first the case when both $u$ and $v$ belong to the path from $i$ to $a$ in the tree $T'$.  For example,  this is always the case when  the path from $i$ to $a$  is not affected by removing the edge $qr$. 
 Then, by definition,  $\mathcal{P}^{ia}_T = \mathcal{P}^{ia}_{T'}$ that is we can use 
 the same sets of perturbations to change labels $i$ and $a$ in both trees.
We  order them arbitrary to form a sequence   $(\swap_1, \ldots, \swap_k)$. 
 Note also that 
for any perturbation $\swap_x^{yz} \in \mathcal{P}^{ia}_T $ we have $d_T(\{y,z\},\{r,s\}) \geq \rho$   due to the property (ii) and $d_T(\{i,a\},\{r,s\}) \geq \rho+1$. 
 Since $F$ is $\rho$-superposable and using properties (iii) and (iv), we get that 
  \[
 	F(  \swap_t \cdots \swap_1 T ) -
 	F(  \swap_t \cdots \swap_1 T' )
 	- F(  \swap_{t+1} \cdots \swap_1 T ) + 
 	F(  \swap_{t+1} \cdots \swap_1T' ) =0.
  \]
Summing up these equalities for all $t =0, \ldots k-1$, we get that
\begin{equation}\label{eqT-T'}
	F(T) - F(T') - F(T^{(ia)}) + F(T'^{(ia)})=0.
\end{equation}

We still need to consider the case when removing $qr$ changes the path from 
$i$ to $a$ such that $u$ or $v$ do not lie on the path anymore.  
In this case,  one have to be slightly more careful with the order of perturbations 
 $(\swap_1, \ldots, \swap_k)$ to avoid the appearance of cycles in 
 $\swap_t \cdots \swap_1 T'$. 
    Without loss of generality we may assume that $d_T(i,q)<d_T(i,r)$ (otherwise, swap the roles of $i$ and $a$). 
 Let $v'$ be the vertex adjacent to $a$ that lies on the path from $i$  to $a$ in $T'$.
 In notations of part (b), we define 
 $\swap_1 = \swap_i^{uv}$ and $\swap_2 = \swap_{v'}^{ai}$, 
 then put the remaining perturbations in any order. A sequence  $(\swap_1, \ldots, \swap_k)$ defined in this way ensures that  all intermediate steps from $T'$ to $T'^{(ia)}$ are trees. Repeating the same argument as above, we prove \eqref{eqT-T'}.
   To complete the proof of part (c), we just need to substitute $T$  by $T^{\omega}$ similarly to part (b).

 Finally, we prove (d) by repeteadly using part (c) for a sequence of perturbations 
 $\swap_{q}^{rs} \in \mathcal{P}_T^{jb}$ 
 that turn  $T$ into $T^{(jb)}$.  
 We can apply part (c)   for all intermediate trees $T'$ because 
 the assumption $d_T(\{i,a\}, \{j,b\}) \geq \rho +2$  
 together with  properties (ii), (iii), (iv)  implies that 
  $d_{T'}(\{i,a\}, \{r,s\}) \geq  \rho +1$.
  
  \end{proof}

 \subsection{Martingale properties} 

Here, we establish the properties of martingales $\Y$ and $\Z(T)$ from \eqref{def_YZ}  
needed to apply the results 
of Section \ref{S:m_theory}.    For a tree  $T\in \Tset_n$ and $A,B\subset [n]$, define 
  \[
  	\one_T^\rho(A,B) := 
  	\begin{cases}
  	 1, & \text{if } d_T(A,B)< \rho,\\
  	 0, & \text{otherwise.}
  	\end{cases}
  \]
%That is,  $ \|T\|_\rho $ is the maximal size of a ball of radius $\rho$ in $T$.
We will repeatedly use the fact that for any $T\in \Tset_n$ and $i \in [n]$, we have
\begin{equation}\label{norm-beta}
 \sum_{j=1}^{n} \one_T^\rho(\{i\},\{j\}) \leq \rho^2 \beta(T),
 \end{equation}
 where $\beta(T)$ is  the parameter defined in \eqref{def_beta}.
In the following, for simplicity of notations, we write $ \one_T^\rho(i,B)$, 
or  $ \one_T^\rho(A,j)$, or $ \one_T^\rho(i,j)$  when $A$,  or $B$, or both are one-element sets.
Let $\Tset_n^d\subset \Tset_n$ be the set  of trees with degrees at most $d$.
We denote by $a\wedge b$ the minimum of two real numbers $a$, $b$.

%Define also
%\[
%	\beta(T, d)= \max_{i \in [n]}
%	\left| \{j \in [n] \st d_T(i,j)\leq d\}\right| 
%\]
%which is  the maximum size of the $d$-neighborhood in the tree $T$. 
%Then, for any $i \in [n]$ and non-increasing $\rho : \Naturals \rightarrow \Reals^+$ with $\rho(n+1)=0$, we have
%\[
%	\sum_{j=1}^{n} \rho(d_T(i,j)) =
%	\sum_{d=1}^{n}\rho(d) \left| \{j \in [n] \st d_T(i,j) = d\}\right| 
%	 \leq 
%	\sum_{d=0}^n (\rho(d)-\rho(d+1)) \beta(T, d).
%\]
\begin{lemma}\label{L:tree-mart}
    Let   $F: \Tset_n\rightarrow \Reals$ be $\alpha$-Lipschitz and  $\rho$-superposable
   for some $\alpha\geq 0$ and $\rho\geq 1$. Then, the following holds  for  all $i\in[n-1]$, 
   $d \in \Reals^+$ and $T \in \Tset_n^d$.
	\begin{itemize}
		\item[(a)]   
		$|Y_i-Y_{i-1}| \leq \ran_{\calF_{i-1}}  [Y_i]  \leq \alpha$.
		
		\item[(b)] $ \displaystyle
		\ran_{\calF_{i-1}} \left[\Var_{\calF_i} [Y_{n-1}]  \right]
		\leq 
		 32\alpha^2 \rho^2  \mysup_{\calF_{i-1}}  \left[ \E_{\calF_{i}} [\beta(T(\U))] \right].
		$
	\item[(c)]
	$\displaystyle
				     |Z_i(T)-Z_{i-1}(T)|\leq  \ran_{\calG_{i-1}} [Z_i(T)] \leq 
				 \max_{\omega,(ia)\in S_n} |F(T^{\omega}) - F(T^{ \omega \circ (ia) })| \leq 
				     2\alpha  d.
                     $
		
		\item[(d)]
		$\displaystyle
			\ran_{\calG_{i-1}}  \left[\Var_{\calG_i} [Z_{n-1}(T)]\right]  \leq 64\alpha^2 d^2  (\rho+2)^2 \beta(T) \log n.
		$
		\item[(e)]  
		Let $V(\uvec) :=  \Vari{Z_{n-1}(T(\uvec))}  =   \Vari {F_{T(\uvec)} (\X)}$. Then, $0\leq V(\U) \leq 4 \alpha^2 n^2$ and
		 \begin{align*}
		 	\ran_{\calF_{i-1}} \Big[ \E_{\calF_i}  &\left[V(\U) \one_{T(\U) \in \Tset_n^d}\right] \Big]
		 	\\
		 	&\leq   \alpha^2
	 	\mysup_{\calF_{i-1}} 
	 	\left[ 		  \E_{\calF_{i}}  \left[4n^2 \one_{T(\U) \notin \Tset_n^d}  + 8d^2 (\rho+1)^2 \beta(T(\U)) \right]\right].
		\end{align*}	
	\end{itemize}
\end{lemma}
\begin{proof}	
 	Using bound  \eqref{eq:bound_range}, we find that  
	\[
		|Y_i - Y_{i-1}|\leq \ran_{\calF_{i-1}} [Y_i] \leq \max |\hat{F}(T(\uvec)) - \hat{F}(T(\uvec'))|,
	\]
	where $\uvec,\uvec' \in [n]^{n-1}$ differ in $i$-th coordinate. Observe that
	\begin{equation}\label{eq_vv}
		T(\uvec') =  \swap_{i+1}^{i \wedge u_i \, i\wedge u_i'} T(\uvec).
	\end{equation}
	From \cref{l:FF}(a), we know that $\hat{F}(T)$ is $\alpha$-Lipschitz.	 Part (a) follows.
	
	As explained in $\eqref{var-as-func}$, we have $Y_i = f_i(\U)$, where 
	\[
		 f_i(\uvec) = \Exp{\hat{F}(T(\U)) \mid u_{\leq i}}
	\]
	and $\E (\cdot \mid u_{\leq i})$ stands for $\E (\cdot \mid U_1=u_1,\ldots,U_i=u_i)$. Let $0\leq i<j \leq  n-1$. Using formula \eqref{eq_vv}, we find that 
	\[
		f_j(\uvec) - f_{j-1}(\uvec) =  
		 \dfrac{1}{n}\sum_{u = 1}^n \Exp{ \hat{F}(T(\U))  - 	\hat{F}\left(\swap_{j+1}^{j\wedge u_j\, j \wedge u}T(\U)
		\right) \mid u_{\leq j}}.
	\]  
	Consider $\uvec'\in [n]^{n-1}$ that differs from $\uvec$ only in $i$-th coordinate.
	 Then, we have 
	 \begin{align*}
	 	f_j(\uvec) - f_{j-1}(\uvec) &- f_{j}(\uvec') +  f_{j-1}(\uvec') 
	 	= \dfrac{1}{n}\sum_{u = 1}^n 
	 	\E \bigg[ \hat{F}(T(\U))  - 	\hat{F}\left(\swap_{j+1}^{j\wedge u_j \, j \wedge u}T(\U) \right) \\&\qquad \qquad -
	 		\hat{F}\left(\swap_{i+1}^{i \wedge u_i \, i\wedge u_i'}T(\U)\right)
	 		+ \hat{F}\left( \swap_{j+1}^{j\wedge u_j \, j \wedge u} \swap_{i+1}^{i \wedge u_i \, i\wedge u_i'} T(\U)\right)
	 		 \mid u_{\leq j} \bigg]
	 \end{align*}
		 From part (a), we have  $0\leq |f_j - f_{j-1}| \leq \alpha$. 
	 Observe also that  if $U_1=u_1, \ldots,U_{j-1}=u_{j-1} $ and $v \in [i]$ then
	 \[
	 d_{T(\U)}(v, \{i \wedge u_i, i \wedge u_i'\}) =d_{T(\uvec)}(v, \{i \wedge u_i, i \wedge u_i'\}).
	 \]  
	 That is, the distance between $v$
	and $\{i \wedge u_i,i \wedge u_i'\}$ 	is completely determined by
	 $u_1, \ldots, u_{j-1}$ and $v$.  
	 From \cref{l:FF}(a), we know that $\hat{F}(T)$ is $\rho$-superposable. Thus, we find that
	 	\begin{align*}
		|(f_j(\uvec) - f_{j-1}(\uvec))^2 &- (f_{j}(\uvec') -  f_{j-1}(\uvec'))^2| 
		\leq 
		 2 \alpha  |f_j(\uvec) - f_{j-1}(\uvec) -f_{j}(\uvec') +  f_{j-1}(\uvec')| 
		\\
		&
		\leq 
		 \dfrac{4\alpha^2}{n}\sum_{u = 1}^n     \one_{T(\uvec)}^\rho (\{j\wedge u_j, j \wedge u\}, \{i\wedge u_i,i\wedge u_i'\}))
%		 \\
%		& \leq 
%		  4\alpha^2 \left(   \one_{T(\uvec)}^\rho(j\wedge u_j, \{i\wedge u_i, i\wedge u_i' \}) + 
%		  		   \one_{T(\uvec)}^\rho(j, \{i\wedge u_i, i\wedge u_i' \}) \right)
%		  		  +
%		  \dfrac{8\alpha^2}{n}   \|T(\uvec)\|_\rho
	\end{align*}
	Using \eqref{norm-beta}, we can bound
	\begin{align*}
	 	\dfrac{1}{n}\sum_{u = 1}^n \E &\left[      \one_{T(\uvec)}^\rho (\{j\wedge u_j, j \wedge u\}, \{i\wedge u_i,i\wedge u_i'\})		  \mid u_{\leq j-1} \right]
		  \\
		  	  &=
		 \dfrac{1}{n^2}\sum_{u = 1}^n \sum_{u_j =1}^n  
		    \one_{T(\uvec)}^\rho (\{j\wedge u_j, j \wedge u\}, \{i\wedge u_i,i\wedge u_i'\})
		  \\ 
		 &\leq 2\cdot  \one_{T(\uvec)}^\rho (j,  \{i\wedge u_i,i\wedge u_i'\}) +
		  \dfrac{2}{n} \sum_{k=1}^{j-1}   \one_{T(\uvec)}^\rho (k,  \{i\wedge u_i,i\wedge u_i'\})
		\\
		  &\leq  
		  2\cdot  \one_{T(\uvec)}^\rho (j,  i\wedge u_i)  +
		  2\cdot  \one_{T(\uvec)}^\rho (j,  i\wedge u_i') +
		  \dfrac{4 }{n} \rho^2 \beta(T(\uvec)).
 	\end{align*}
	Similarly to \eqref{var-as-func}, let $\ran_{\calF_{i-1}} \left[\Var_{\calF_i} [Y_{n-1}]\right]   = r(U_1,\ldots,U_{i-1})$.
	Using \eqref{telescope2},  \eqref{eq:bound_range} and taking the conditional expectation
	given $U_1= u_1,\ldots,U_{i-1}=u_{i-1}$  for the bounds above, we obtain that
	\begin{align*}
		r(u_1,\ldots,u_{i-1})
		&= \max _{u_i, u_i' \in [n]} 
		\bigg|\sum_{j=i+1}^n \E\left[ (f_j(\U) - f_{j-1}(\U))^2  \mid  u_{\leq i-1}, U_i = u_i\right]
		\\ &\qquad \qquad  -\sum_{j=i+1}^n \E\left[ (f_j(\U) - f_{j-1}(\U))^2  \mid  u_{\leq i-1}, U_i = u'_i\right] \bigg|
		\\
		&
		  \leq    \dfrac{16\alpha^2}{n}  
		    \max_{u \in [n]}
		   \sum_{j=i+1}^n
				       \E  \left[ \one_{T(\uvec)}^\rho (j,  i\wedge u)  + \rho^2 \beta(T(\U)) \mid  u_{\leq i-1},  U_i = u \right] 
 		 \\ 
 		  &\leq 
 		  32\alpha^2 \rho^2 \max_{u_i \in [n]} \E (\beta(T(\U)) \mid  u_{\leq i}).
	\end{align*}
	This completes the proof part (b).

	Part (c) immediately follows from Lemma \ref{Lemma_perm}(a) and \cref{l:FF}(b). Indeed, 
	 \begin{equation}\label{eq_alpha}
	 	\alpha_i[F_T]  \leq \max_{\omega, (ia)\in S_n} 
	 	|F_T(\omega) - F_T(\omega\circ (ia))| \leq 2\alpha d. 
	 \end{equation}
 		For  (d),  recall from \eqref{eq:CLTbounds} that
 	\begin{equation}\label{Var_split}
 		 \ran_{\calG_{i-1}} \left[ \Var_{\calG_i} [Z_{n-1}(T)]\right]
 		  \leq \sum_{j=i+1}^{n-1} \ran_{\calG_{i-1}} \left[ \E_{\calG_i} \left[(Z_{j}(T) - Z_{j-1}(T))^2\right] \right].
 	\end{equation}
 	 We  will apply  Lemma \ref{Lemma_perm}(b) to estimate the right-hand side of \eqref{Var_split}.
 	 From Lemma \ref{l:FF}(d) and the bound \eqref{eq_alpha}, we get that
 	 \[
 	 	 |F_T(\omega) - F_T(\omega\circ(ia)) - F_T(\omega\circ(jb)) + F_T(\omega \circ(jb)\circ(ia)) |
 	 	 \leq 4 \alpha d \, \one_{T}^{\rho+2}(\{i,a\}, \{j,b\}).
 	 \]	 
 	 	Bounding
 	\[\one_{T}^{\rho+2}(\{i,a\}, \{j,b\}) \leq 
 	 \one_{T}^{\rho+2}(i,j) + \one_{T}^{\rho+2}(i,b) + \one_{T}^{\rho+2}(a,j) + \one_{T}^{\rho+2}(a,b)
 	\]
 	and using \eqref{norm-beta}, we find that, 
 	 	 	for $1\leq i <j \leq n-1$,	
 	\begin{align*}
 		\varDelta_{ij} [F_T] \leq 4 \alpha d
 		\sum_{a=i+1}^n \sum_{b=j+1}^n \frac{ \one_{T}^{\rho+2}(\{i,a\}, \{j,b\})}{(n-i)(n-j)} 
 		\leq 4 \alpha d \left( \one^{\rho+2}_T(i,j) +  \frac{3 (\rho+2)^2 \beta(T)}{n-j} \right).
 	\end{align*} 
 	Combining \eqref{norm-beta}, \eqref{eq_alpha}, Lemma \ref{Lemma_perm}(b) and the inequality 
 	\[ 1+3 \sum_{k=1}^{n-1} k^{-1} \leq 4+ 3 \log n \leq 4\log n,
 	\]
 	 we obtain that
 	\begin{align*}
 	   \ran_{\calG_{i-1}} \left[ \Var_{\calG_i} [Z_{n-1}(T)] \right]
 	    &\leq \sum_{j=i+1}^{n-1}    16 \alpha^2 d^2  \left( \one^{\rho+2}_T(i,j) +  \frac{3(\rho+2)^2 \beta(T)}{n-j} \right)
 	    \\&\leq 64 \alpha^2 d^2 (\rho+2)^2 \beta(T) \log n.
 	\end{align*}

 	 Finally, we proceed to part (e). Since $F$ is $\alpha$-Lipschitz, 
 	 we have 
 	 $|F(T)  - F(T')| \leq 2 \alpha n$   for any two trees $T,T'\in \Tset_n$.  Indeed, applying at most $n$ 
 	 perturbations of type $\swap_{x}^{y 1}$, where $x$ is a leaf,  we can turn any tree into a star centered at vertex $1$.
 	 Thus, we can bound
 	 \[
 	 	 0 \leq V(\uvec)\leq  4 \alpha^2 n^2.
 	 \]
 	 Then, for any $\mathcal{A} \subset [n]^{n-1}$ and $u_1, \ldots, u_{i-1} \in [n]$,
 	  	\begin{align*}
 	    &\ran \Big[ \E_{\calF_i} \left[V(\U)\one_{\U \in \mathcal{A}}\right] \mid  u_{\leq i-1} \Big] 
 	    \\
 	    &=  
 	         \max_{u \in[n]} \E \left[V(\U)\one_{\U \in \mathcal{A}}  \mid  u_{\leq i-1}, U_i =u\right] -
 	         \min_{u \in [n]}  \E \left[V(\U)\one_{\U \in \mathcal{A}}  \mid  u_{\leq i-1}, U_i =u\right]
 	          \\
 	           &\leq 
 	          4 \alpha^2 n^2    \max_{u_i \in[n]} \Pr(\U \notin \mathcal{A} \mid u_{\leq i})
 	         %\\&\qquad \qquad \qquad \qquad
 	           +   \max_{u_i,u \in [n]} 
	          \E \left[  (V(\U) -  V(\U')) \one_{\U,\U' \in \mathcal{A}}
	          \mid  \uvec_{\leq i},  U_i' =u \right] 
 	\end{align*}
 	where $\U'$ differs from $\U$ in $i$-th coordinate only.
 	For the following we put $\mathcal{A} = \{\uvec \in [n]^{n-1} \st T(\uvec) \in \Tset_n^d\}$. It remains to bound 
 	$V(\U)-  V(\U')$ when $T(\U),T(\U') \in \Tset_n^d$.  
 	
 	  	Consider any $\uvec,\uvec' \in [n]^{n-1}$  that differ in $i$-th coordinate only 
 	  	and  $T(\uvec), T(\uvec') \in \Tset_n^d$.  	If   $T(\uvec) = T(\uvec')$ then $V(\uvec) =  V(\uvec')$. 
 	Otherwise,  
 	 recalling \eqref{eq_vv}, we can find some relabelling $\sigma \in S_n$ 
 	  that  the trees $T = T(\uvec)^{\sigma}$, $T'=T(\uvec')^{\sigma}$ satisfy 	 $T' = \swap_{3}^{12} T$  and
 	\[
 		0 = d_T(1, \{1,2\}) \leq \cdots \leq d_T(n, \{1,2\}).  
 	\]
 	 	 Note that $\Var [F_T(\X)] = V(\uvec)$ and $\Var [F_{T'}(\X)] = V(\uvec')$. 	 
 Using Lemma \ref{l:FF}(c) and \eqref{eq_alpha},  we find that, 	for any $1\leq i <a\leq n$, 
 	 \begin{align*}
 	 	 |F_T(\omega) - F_T(\omega\circ(ia)) - F_{T'}(\omega) + F_{T'}(\omega \circ(ia)) |
 	 	 &\leq 4 \alpha d \, \one_{T}^{\rho+1}(\{i,a\}, \{1,2\})
 	 	 \\ &\leq 4 \alpha d \, \one_{T}^{\rho+1}(i, \{1,2\}).
 	 \end{align*}
 	Applying Lemma \ref{Lemma_perm}(a) to the function	$F_T - F_{T'}$, we obtain
 	\begin{align*}
 		|Z_i(T) - Z_{i-1}(T) - Z_i(T') + Z_{i-1}(T')| 
 		\leq  
 		\alpha_i(F_T - F_{T'}) \leq 4 \alpha d \, \one_{T}^{\rho+1}(i, \{1,2\}).
  	\end{align*}
 	We have already proved in part (b) that   $|Z_i(T) - Z_{i-1}(T)| \leq 2 \alpha d$.
 	  Using \eqref{telescope} and \eqref{norm-beta}, we  bound
 	\begin{align*}
 	V(\uvec) - V(\uvec')
 		&= \Vari{Z_{n-1}(T)}-  \Vari {Z_{n-1}(T')} \\
 		 &= \sum_{i=1}^{n-1}  \E \left[ ( Z_i(T) - Z_{i-1}(T))^2  - (Z_i(T') - Z_{i-1}(T'))^2\right]
 		\\ &\leq \sum_{i=1}^{n-1} 4\alpha^2 d^2 \one_{T}^{\rho+1}(i, \{1,2\}) \leq 8 \alpha^2d^2  (\rho+1)^2 \beta(T). 
 	\end{align*}
Part  (e) follows. 
\end{proof}

\subsection{Proof of Theorem \ref{T:main}}\label{S:main}

Before proving  of Theorem \ref{T:main}, we need one more lemma. Let
\begin{align*}
	\mathcal{U}_{\rm small} &:= \{\uvec\in [n]^{n-1} \st T(\uvec) \in \Tset_{n}^{\log n} \text{ and } \beta(T(\uvec)) \leq   \log^4 n\},
\\
	\mathcal{U}_{\rm big} &:= \{\uvec\in [n]^{n-1} \st T(\uvec) \in \Tset_{n}^{ 2 \log n} \text{ and } \beta(T(\uvec)) \leq  2 \log^4 n\}.
\end{align*}
\begin{lemma}\label{l:small+big}
The following asymptotics bounds hold   for any $\uvec \in \mathcal{U}_{\rm small}$,  $u \in [n]$:
\[
	\Pr(\U \notin  \mathcal{U}_{\rm small}) = e^{-\omega(\log n)}, \qquad 
	\Pr(\U \notin  \mathcal{U}_{\rm big} \mid u_{\leq i-1}, U_i =u) = e^{-\omega(\log n)}.
\]
\end{lemma}
\begin{proof}
The first bound follows immediately from \eqref{eq:Moon} and \cref{T:beta}.
For the second, observe that, for any $u_1',\ldots, u_{i-1}' \in [n]$,  
\[
	 \Pr \left(\U \in  \mathcal{U}_{\rm small} \mid 
	U_1 = u_1', \ldots, U_{i-1} = u_{i-1}'
	 \right)
	 \leq 
	 \Pr \left(\U \in \mathcal{U}_{\rm big} \mid u_{\leq i-1}
	 \right).
\]
Indeed, let  $\U$, $\U'$ are such that 
$U_j = u_j$ and $U_j' = u_j'$ for $j \in [i-1]$ and $U_j = U_j'$ for $j \geq i$.
Then,  $T(\U) \subset T(U') \cup T(\uvec)$ because the edges corresponding from $i-1$ steps of the  Aldous -Broder algorithm for $T(\U)$ lie in $T(\uvec)$,  while the remaining edges are covered by $T(U')$). We know that $\uvec \in \mathcal{U}_{\rm small}$. Therefore, if  $\U' \in  \mathcal{U}_{\rm small}$ then  $\U \in \mathcal{U}_{\rm big}$. 
	 
Next, averaging over all 	 $u_1',\ldots, u_{i-1}' \in [n]$, we conclude that 
\[
	\Pr \left(\U \notin \mathcal{U}_{\rm big} \mid u_{\leq i-1}
	 \right) \leq  \Pr \left(\U \notin \mathcal{U}_{\rm small}\right).
\]
Note that, 
for any $u \in [n]$, 
\begin{align*}
	\Pr(\U \notin \mathcal{U}_{\rm big} \mid  u_{\leq i-1}, 
	U_i = u) 
	= \frac{\Pr(\U \notin \mathcal{U}_{\rm big}, U_i = u\mid  u_{\leq i-1})}
	{\Pr (U_i = u \mid u_{\leq i-1})}  
	\leq n  \Pr \left(\U \notin \mathcal{U}_{\rm small}
	 \right).
	   %=	 e^{-\omega(\log n)}.
\end{align*}
Recalling   $ \Pr \left(\U \notin \mathcal{U}_{\rm small}\right) = e^{-\omega(\log n)}$, we complete the proof.
\end{proof}

Now we are ready to prove  Theorem \ref{T:main}, our main result. 
Let $\Y$ and $\Z(T)$ be the martingales from \eqref{def_YZ}.
Consider the  sequence  $\boldsymbol{W} = (W_0, \ldots, W_{2n-2})$  defined by
\[
W_i :=   
\begin{cases}
  Y_i, &\text{if } i =0,\ldots, n-1,
  \\
  Z_{i-n+1}(T(\U)), & \text{if  $i \geq n$  and $T(\U) \in \Tset_n^{\log n}$,} 
  \\
    Y_{n-1}, & \text{if  $i \geq n$  and $T(\U) \notin \Tset_n^{\log n}$.}
\end{cases}
\]
Note that $\boldsymbol{W}$ is a martingale with  the respect to the filtration $\calF'_0, \ldots, \calF'_{2n-2}$,  
where 
$
\calF'_i = \calF_i,
$
for $i\leq n-1$ and 
$
	\calF'_i = \calF_{n-1} \times \calG_{i-n+1},
$
for $i \geq n$.
Using   \eqref{eq:Moon},  \eqref{telescope}, and \cref{L:tree-mart}(e), we get that
\begin{align*}
	\Vari{W_{2n-2}} &=  \Vari {Y_{n-1}} + \E \left[  V(\U) \one_{T(\U) \in  \Tset_n^{\log n}} \right]
	\\&=   \Vari {Y_{n-1}}  + \Exp{ V(\U)} -  
	 4\alpha^2 n^2 e^{-\omega(\log n)}
	= \Vari {F(\Tran)} -  \alpha^2 e^{-\omega(\log n)}.
\end{align*}
Then, by assumptions of Theorem \ref{T:main}, we get   $\Vari{W_{2n-2}}  =  \left(1+ e^{-\omega(\log n)}\right) 
\Vari {F(\Tran)} $ 
and 
\begin{equation}\label{main:ass}
  \alpha^2 = O\left(n^{-2/3 - 2\varepsilon/3}\right)  \Vari {W_{2n-2}}, \qquad 
	\alpha^2 \rho^2 = O\left(n^{-1/2 - 2\varepsilon}\right) \Vari {W_{2n-2}}. 
\end{equation}
  Using  Lemma \ref{L:tree-mart}(a,c),   we obtain that,  for all $i\in [2n-2]$,
 \begin{equation}\label{YZ-step}
 	 W_i - W_{i-1} =  O(\alpha \log n).
 \end{equation}
% For the following bounds  we assume that  $\U \in \mathcal U_{\rm small}$. This event  happens with probability $1 - e^{-%\omega(\log n)}$ by \cref{l:small+big}.
 Let $\uvec \in [n]^{n-1} \in U_{\rm small}$. Combining Lemma  \ref{L:tree-mart}(b,d,e) and \cref{l:small+big}
  and observing $\beta(T)\leq n^2$ for all $T \in \Tset_n$,   we get that, for all $i\in [n-1]$
 \begin{align*}
 	\ran\left[ \Var_{\calF_i} [Y_{n-1}] \mid u_{\leq i-1} \right] &= O(\alpha^2 \rho^2 \log^4 n) \\
 	\ran_{\calG_{i-1}}[ Z_i(T(\uvec))] &= O(\alpha^2 \rho^2 \log^7 n) \\
 	\ran\left[ \E_{\calF_i} [V(\U)] \mid u_{\leq i-1} \right] &= O(\alpha^2 \rho^2 \log^6 n) 
 \end{align*}
 Note that, in the case of the event  $\U \in \mathcal{U}_{\rm small}$, 
  we have $W_i = Z_i(T(\U))$ and
  \[ 
    \Var [W_{2n-2} \mid \calF_{i}] = \Var_{\calF_i} [Y_{n-1}] + \E_{\calF_i} [V(\U)].
    \]
 Then, we obtain that if $\U \in \mathcal{U}_{\rm small}$ then, for all $i \in [2n-2]$,
 \[
 	\ran\left[  \Var [W_{2n-2} \mid \calF'_{i}] \mid \calF'_{i-1}\right] = O(\alpha^2 \rho^2 \log^7 n).
 \]
 Using \eqref{main:ass}, we conclude that,  with probability $1-e^{-\omega(\log n)}$,
 \begin{align*}
 	\sum_{i=1}^{2n-2} &\left( \ran\big[  \Var [W_{2n-2} \mid \calF'_{i}] \mid \calF'_{i-1}\big]  + 
 	 	 \big(\ran\left[ W_i \mid \calF'_{i-1}\right]\big)^2 
 	\right)^2 
 	\\&\qquad=
  	 O(\alpha^4 \rho^4 n \log^{14} n)  
  	 = O(n^{-4\eps}\log^{14} n) \left( \Var [W_{2n-2}]\right)^2.
 \end{align*}
 Let $\tilde{\eps} \in (0,\eps)$.  Setting $\hat{q}= n^{-2\tilde{\eps}}$ 
 and applying \cref{l:CLTbounds}, we get that,  for any $p \in[1,+\infty)$,
\begin{align*}
	 \Exp {|Q[\boldsymbol{W}]-1|^p}  =O\left( n^{-2\tilde{p\eps}} + \sup| Q[\boldsymbol{W}]-1|^p e^{-\omega(\log n)} \right).
\end{align*}
Using \eqref{YZ-step} and \eqref{main:ass}, we can bound
\[
 Q[\boldsymbol{W}] =  
  \frac{1}{\Var [W_{2n-2}]}\sum_{i=1}^{2n-2} (W_i - W_{i-1})^2 
   = O(n^{1/3}). 
  \]
  Applying \cref{T:distr} to the scaled martingale sequence $\W / (\alpha \log n)$, we get that
  \begin{align*}
  	\delta_{K}[W_{2n-2}] 
  	&= O\left( \left( \frac{\alpha^2 \log^2 n  }{  \Var [W_{2n-2}]}\right)^{3/2} \hspace{-2mm} n \log n  
  	 +  \left(n^{-2\tilde{p\eps}} + e^{-\omega(\log n)} n^{p/3}  \right)^{1/(2p+1)}
  	\right) \\
  	&= O\left(n^{-\eps} \log^4 n + n^{-2p \tilde{\eps}/(2p+1) }\right) = O(n^{-2p \tilde{\eps}/(2p+1) }).
  \end{align*}
  We can make  $2p \tilde{\eps}/(2p+1) \geq \eps'$ for any $\eps'\in (0,\eps)$ by taking  
  $\tilde{\eps}$ to be  sufficiently close to $\eps$ and  $p$ to be sufficiently large. 
  Recalling that $W_{2n-2} = F(T(\U)^{\X})$  
  with probability  $1 - e^{\omega(\log n)}$ (that is for the event   $T(\U) \in \Tset_n^{\log n}$)
   and $\Var [W_{2n-2}]  =  \left(1+ e^{-\omega(\log n)}\right) \Vari{F(\Tran)}$,  
   the required bound for $\delta_{K}[F(\Tran)]$ follows.

\begin{remark}
	The proof of Theorem \ref{T:main} can be significantly simplified under additional assumption that the tree parameter $F$ is symmetric. Namely, we would not need  the martingale sequence $\Z(T)$, the bounds of \cref{S:FF}, and we would only use parts (a), (b)  from \cref{L:tree-mart}.  In fact, a symmetric version  of  Theorem \ref{T:main} would be sufficient to cover all aplications given in Sections \ref{S:patterns} and \ref{S:auto}. 	Our decision to consider arbitrary tree parameters serves two purposes. First, the result is significantly stronger.  Second, the analysis of martingales based on functions with dependent random variables is essential for extensions to more sophisticated tree models.
\end{remark}

\begin{remark}
	Combining \cref{L:tree-mart}(a,c) and \cref{T:concentration} one can easily derive fast decreasing bounds 
	for the tail of the  distribution of $F(\Tran)$, provided a tree parameter $F$ is  $\alpha$-Lipschitz.  
	Cooper, McGrae and Zito \cite[Section 4]{CMZ2009} used a different martingale 
	construction  for trees to establsih the concentration of $F(\Tran)$ around its expectation, however they needed more	restrictive assumptions 
	about the tree parameter $F$.
\end{remark}

%
%\section{Appendix}\label{S:appendix}
%Here we collect the technical lemmas that are used in the proofs. This section is
%self-contained and does not rely on assumptions other than those stated.
%
%\subsection{Extension lemma}\label{S:extension}

 %%%%%%%%%%%%%%%%%%%%%%%%%%%%%%%%%%%%%%%%%%%%%%%%%%%%%%%%%%%%
%%%%%%%%%%%%%%%

 %%%%%%%%%%%%%%%%%%%%%%%%%%%%%%%%%%%%%%%%%%%%%%%%%%%%%%%%%%%%
%%%%%%%%%%%%%%%%%%%%%%%%%%%%%%%%%%%%%%%%%%%%%%%%%%%%%%%%%%%%
%%%%%%%%%%%%%%%%%%%%%%%%%%%%%%%%%%%%%%%%%%%%%%%%%%%%%%%%%%%%
%%%%%%%%%%%%%%%%%%%%%%%%%%%%%%%%%%%%%%%%%%%%%%%%%%%%%%%%%%%%
%%%%%%%%%%%%%%%%%%%%%%%%%%%%%%%%%%%%%%%%%%%%%%%%%%%%%%%%%%%%
%%%%%%%%%%%%%%%%%%%%%%%%%%%%%%%%%%%%%%%%%%%%%%%%%%%%%%%%%%%%

 \section{The balls in random trees are not too large}\label{S:balls}

 In this section we prove Theorem~\ref{T:beta} using martingales. 
 For a tree $T\in \Tset_n$,  let  $\Gamma^k_{T}(v)$ be the set of all vertices at distance exactly $k$ from $v$. Theorem~\ref{T:beta} follows immediately from \cref{L:X'} and \cref{balls_final_lemma} (stated below) by 
 summing over all $|\Gamma_{\Tran}^k(v)|$ for $k=1,\ldots, d$ and  using the union bound over all vertices $v \in [n]$.

Let $a>b$ be positive integers. Let $A$ be an arbitrary set of $a$ vertices from $[n]$, and $B$ be its subset on $b$ vertices. Consider event $\mathcal{E}_{A,B}$ that $A$ induces a tree and vertices of $A\setminus B$ have neighbors only in $A$. For $T\in\mathcal{T}_n$, let $\xi_{A,B}(T)$ be the number of neighbors of $B$ in $T$ outside $A$. Below, we denote the random variable $\xi_{A,B}(\Tran)$ simply $\xi_{A,B}$. 

\begin{lemma}
The conditional distribution of $\xi_{A,B}-1$ subject to  $\mathcal{E}_{A,B}$ is binomial with parameters $(n-a-1,\frac{b}{n-a+b})$.
\label{bin_leafs}
\end{lemma}

\begin{proof} 
Let $T_0$ be a tree on $A$. Consider event $\mathcal{E}_{A,B,T_0}$ that $A$ induces exactly the given subtree $T_0$ and vertices of $A\setminus B$ have neighbors only in $A$. By Lemma~\ref{L:forest},
$$
 \left|\mathcal{E}_{A,B,T_0}\right|=b(n-a+b)^{n-a-1}.
$$

Let $k\in\mathbb{N}$.  By Lemma~\ref{L:forest},
$$
 \left|\{\xi_{A,B}=k\}\cap\mathcal{E}_{A,B,T_0}\right|=b^k{n-a-1\choose k-1}(n-a)^{n-a-k}.
$$

Therefore,
\begin{align*}
 \Prob(\xi_{A,B}=k \mid \mathcal{E}_{A,B})&=\frac{\sum_{T}\Prob(\xi_{A,B}=k \mid \mathcal{E}_{A,B,T})\Prob(\mathcal{E}_{A,B,T})}{\sum_{T}\Prob(\mathcal{E}_{A,B,T})}
 \\ &=\Prob(\xi_{A,B}=k \mid \mathcal{E}_{A,B,T_0})
 \\&=
 \left(\frac{b}{n-a+b}\right)^{k-1}\left(1-\frac{b}{n-a+b}\right)^{n-a-k}{n-a-1\choose k-1},
\end{align*}
which  is the required distribution.
\end{proof}

% Let us compute the number of trees on $\{1,\ldots,n\}$ such that the conditions * and ** hold and 
%\begin{itemize}
%\item[***] vertices from $B$ have exactly $k$ neighbors outside $A$. 
%\end{itemize}

%Let $T$ on $\{1,\ldots,n\}$ be a tree satisfying * and **. Consider the tree $\mathrm{R}(T)$ on the vertex set $\tilde A=\{A\}\cup \{1,\ldots,n\}\setminus A$ of size $n-a+1$ having the same adjacencies between vertices of $\{1,\ldots,n\}\setminus A$ as $T$ has and having $A\sim i$ for $i\in\{1,\ldots,n\}$ if and only if $i\sim j$ in $T$ for a certain $j\in B$. Clearly, for a tree $S$ on $\tilde A$, there are $b^{\mathrm{deg}_{S}(A)}$ trees $T$ satisfying * and ** such that $S=\mathrm{R}(T)$. Therefore, by Lemma~\ref{degree_sequence}, there are
%$$
 %\sum_{x_1,\ldots,x_{n-a}}b^k{n-a-1\choose k-1,x_1\ldots,x_{n-a-1}}=b^k{n-a-1\choose k-1}(n-a)^{n-a-k}
%$$
%trees on $\{1,\ldots,n\}$ such that $A$ induces $T_0$.

%Let $\mathcal{C}_A$ be the event that $T_n$ satisfies * and **, and $\xi_A$ be the number of neighbors of vertices from $B$ outside $A$. From the above, we get
%$$
% {\sf P}(\xi_A=k|\mathcal{C}_A)=\left(\frac{b}{n-a+b}\right)^{k-1}\left(1-\frac{b}{n-a+b}\right)^{n-a-k}{n-a-1\choose k-1}
%$$
%Therefore, $\xi_A-1$ has Bin$(n-a-1,\frac{b}{n-a+b})$ distribution conditioning on $\mathcal{C}_A$. \\

Fix a vertex $v\in [n]$. Define  the sequence of  random variables $X_0,\ldots, X_n$ by 
\[
  X_0:=1, \qquad \text{and} \qquad X_k:=|\Gamma^k_{\Tran}(v)| \text{ for all $k \in [n]$}.
  \]
From Lemma~\ref{bin_leafs}, we have $X_1-1\sim$Bin$(n-2,\frac{1}{n})$.
Notice that, for $k\geq 1$, the vertices of $\Gamma^{k+1}_{\Tran}(v)$ are adjacent only to the vertices of $\Gamma^k_{\Tran}(v)$ in $\bigsqcup_{j\leq k+1}\Gamma^j_{\Tran}(v)$. Let $(x_1,\ldots,x_k)$ be
a sequence of positive integers such that $1+x_1+\ldots+x_k\leq n$. By Lemma~\ref{bin_leafs}, if $x_1+\ldots+x_k\leq n-3$, then 
the conditional distribution of 
$X_{k+1}-1$  subject to $(X_1=x_1,\ldots,X_k=x_k)$ is  binomial with parameters $n-x_1-\ldots-x_k-2$ and $\frac{x_k}{n-x_1-\ldots-x_{k-1}-1}$. If $x_1+\ldots+x_k=n-2$, then $X_{k+1}=1$. Finally, if $x_1+\ldots+x_k=n-1$, then $X_{k+1}=0$.

\begin{lemma}\label{L:X'}
There exists a sequence $X_0=X_0',X_1',\ldots,X_n'$ such that
\begin{itemize}
\item $X_k'\geq X_k$,
\item for $k\geq 0$, the distribution of $X_{k+1}'-1$ subject to $X_j=x_j,X_j'=x_j'$, $j\in[k]$, is 
\[
\begin{cases}
\mathrm{Bin}\left(n-\sum_{j=0}^{k-1} x_j,\frac{x_k'}{n-\sum_{j=0}^{k-1} x_j}\right), &\text{if }n-\sum_{j=0}^{k-1} x_j\geq x_k',\\
x_k' \text{ with probability } 1,  &\text{otherwise}.
\end{cases}
\]

%$(x_1,\ldots,x_{k-1},y_k)\in\mathbb{Z}_+^{k-1}\times\mathbb{Z}$ such that $y_k+k\leq n-x_1-\ldots-x_{k-1}-1$, the distribution of $Y_{k+1}+k$ conditioning on $(X_1=x_1,\ldots,X_k=x_{k-1},Y_k=y_k)$ is 
%$$
% \mathrm{Bin}\left(n-X_1-\ldots-X_{k-1}-1,\frac{Y_k+k}{n-X_1-\ldots-X_{k-1}-1}\right),
%$$
%\item for $k\in\mathbb{N}$ and a vector $(x_1,\ldots,x_{k-1},y_k)\in\mathbb{N}^k$ such that $y_k+k> n-x_1-\ldots-x_{k-1}-1$, 
%$$
%{\sf P}(Y_{k+1}=y_k|X_1=x_1,\ldots,X_k=x_{k-1},Y_k=y_k)=1.
%$$ 
%$$
% \mathrm{Bin}\left(n-X_1-\ldots-X_{k-1}-1,\frac{Y_k+k}{n-X_1-\ldots-X_{k-1}-1}\right),
%$$
\end{itemize}
\label{prime_approximation}
\end{lemma}

\begin{proof} 
It is straightforward since, for every $k$, we preserve the denominator of the second parameter of the binomial distribution but make the first one larger.
\end{proof}

Note that $(X_k'-k)_{k\in[n]}$ is a martingale sequence. 
Unfortunately, we can not  apply Theorem~\ref{T:concentration} directly because  every $X_k'$ ranges in a large interval (mostly for small $k$).  Instead, we  cut the tails of these random variables and construct a new martingale. To do  that  we need the following property of binomial distributions.

\begin{lemma}
Let $N$ and $a\leq N$ be positive integers, $\xi\sim\mathrm{Bin}(N,\frac{a}{N})$. Then, for every $b\in\mathbb{N}$, there exists an interval $\mathcal{I}=\mathcal{I}(N,a,b)\subset[a-b,a+b]$ such that 
\begin{itemize}
\item $\Prob(\xi\notin\mathcal{I})\leq N^2\Prob(\xi\notin[a-b,a+b])$, 
\item $\exists c\in[a-b,a+b]$ such that the function $f:\mathbb{R}\to\mathbb{R}$ defined by 
$$
f(x):=\left\{\begin{array}{cc}x, & x\in\mathcal{I} \\ c, & x\notin\mathcal{I}\end{array}\right.
%x \one_\mathcal{I}+cI (x\notin\mathcal{I})
$$ 
satisfies $\Exp {f(\xi)}=a$.
\end{itemize}
\label{cut_bin}
\end{lemma}

\begin{proof}
For $a=N/2$, we get the result by setting $\mathcal{I}=[a-b,a+b]$ and $c=a$.
For the following,  without loss of the generality, we may assume $a<N/2$ since the proof for $a>N/2$ is symmetric.

 Let us consider the set $\mathcal{S}$ of all integers $s$ such that
\begin{equation}
 \E\biggl[\xi \one_{\{\xi\in[a-s,a+b]\}}\biggr]\geq a\Prob(\xi\in[a-s,a+b]).
\label{in_S}
\end{equation}
It is clear that $0\in\mathcal{S}$. However, for every $x\in\mathbb{N}$, $\Prob(\xi=a-x)>\Prob(\xi=a+x)$. Indeed,
$$
  \frac{\Prob(\xi=a-x)}{\Prob(\xi=a+x)}=\frac{(1+\frac{x}{a})(1+\frac{x-1}{a})\ldots(1-\frac{x-1}{a})}{(1+\frac{x}{N-a})(1+\frac{x-1}{N-a})\ldots(1-\frac{x-1}{N-a})}>1.
$$
Therefore, $b\notin\mathcal{S}$. Let $s^*$ be the maximum integer from $\mathcal{S}$. Then,  $s^*\in[1,b-1]$ and
\begin{equation}
\E\biggl[\xi \one_{\{\xi\in[a-s^*-1,a+b]\}}\biggr]<a\Prob(\xi\in[a-s^*-1,a+b]).
\label{bin_threshold}
\end{equation}
Let us prove that $\mathcal{I}=[a-s^*,a+b]$ is the desired interval.
From~(\ref{bin_threshold}), we get 
\begin{align*}
 \E&\biggl[(a-s^*-1)\one_{\{\xi\notin\mathcal{I}\}}+\xi\one_{\{\xi\in\mathcal{I}\}}\biggr]
\\
 &=\E\biggl[(a-s^*-1)] \one_{\{\xi\notin[a-s^*-1,a+b]\}}+\xi\one_{\{\xi\in[a-s^*-1,a+b]\}}\biggr]
\\ 
 &<(a-s^*-1)\Prob(\xi\notin[a-s^*-1,a+b])+a\Prob(\xi\in[a-s^*-1,a+b])<a.
\end{align*}
Moreover, since~(\ref{in_S}) holds for $s=s^*$,
$$
 \E\biggl[a\one_{\{\xi\notin\mathcal{I}\}}+\xi\one_{\{\xi\in\mathcal{I}\}}\biggr]\geq a\Prob(\xi\notin\mathcal{I})+a\Prob(\xi\in\mathcal{I})=a.
$$
Therefore, there exists $c\in(a-s^*-1,a]$ such that $\E[cI(\xi\notin\mathcal{I})+\xi I(\xi\in\mathcal{I})]=a$.

It remains to estimate $\Prob(\xi\notin\mathcal{I})$ from above. Notice that, from~(\ref{bin_threshold}),
$$
 a\Prob(\xi\in[a-s^*-1,a+b])+(a-s^*)\Prob(\xi<a-s^*-1)+N\Prob(\xi>a+b)>a.
$$
Therefore, $s^*\Prob(\xi<a-s^*-1)<N\Prob(\xi>a+b)$. Since $2a\Prob(\xi=a-s^*-2)>\Prob(\xi=a-s^*-1)$, we get 
$$
 \Prob(\xi<a-s^*)<(2a+1)\Prob(\xi<a-s^*-1)\leq N^2\Prob(\xi>a+b),
$$
and this immediately implies that $\Prob(\xi\notin\mathcal{I})\leq N^2\Prob(\xi\notin[a-b,a+b])$.
\end{proof}

%Notice that, by symmetry reasons, an analogue of Lemma~\ref{cut_bin} holds for $a\geq N/2$: for every $k\in\mathbb{N}$, there exist a positive integer $\ell=\ell^+(a,N,k)\leq k$, a real number $c\in[a,a+\ell+1)$ and a function $f=f[a,N,k]:\mathbb{R}\to\mathbb{R}$ such that $f(x)=x$ for $x\in[a-k,a+\ell]$, $f(x)=c$ for $x\notin[a-k,a+\ell]$ and ${\sf E}f(\xi)=a$.\\

Now, we are ready to construct a martingale sequence  that coincides with $X_k'-k$ with  probability very close to $1$, but is more suitable for applying  Theorem~\ref{T:concentration}.
For every $k\geq 2$, consider the event
$$
\mathcal{B}_k:=\left\{n-\sum_{j=0}^{k-2}X_j\geq X'_{k-1}\right\}.
$$
For $\omega\in\mathcal{B}_k$, denote 
\begin{align*}
\mathcal{I}_k&:=\mathcal{I}\left(n-\sum_{j=0}^{k-2}X_j,X_{k-1}',\sqrt{X_{k-1}'}\log n\right),
\\
f_k&:=f\left(n-\sum_{j=0}^{k-2}X_j,X_{k-1}',\sqrt{X_{k-1}'}\log n\right).
\end{align*}
Let
$$
\mathcal{E}_k:=\mathcal{B}_k\cap\left(\bigcap_{j=1}^k\{X'_j-1\in\mathcal{I}_j\}\right).
$$
Define the sequence $(Y_k)_{k\in[n]}$ as follows.
Let $Y_0:=X'_0=1$. For $k\geq 1$, set 
$$
Y_k:=[f_k(X_k'-1)-(k-1)]\one_{\mathcal{E}_k}+Y_{k-1}\one_{\overline{\mathcal{E}_k}}.
$$
Using Lemma~\ref{prime_approximation} and Lemma~\ref{cut_bin}, we find that $(Y_0,Y_1,\ldots,Y_n)$ is a martingale sequence with respect to the filtration $\calF_i = \sigma(X_j,X_j' \st \,0\leq j\leq i)$ for all  $i\in\{0,1,\ldots,n\}$.  
\begin{lemma} 
\label{balls_final_lemma} Let $c>0$ be a fixed constant. Then, the following bounds hold:
\begin{enumerate}
\item[$(a)$] $\Prob(\exists k\in[n]:\,\,Y_k>k\log^4 n)\leq e^{-\omega(\log n)}$,
\item[$(b)$] $\Prob(\exists k\in[n]:\,\,Y_k\neq X_k'-k)\leq e^{-\omega(\log n)}$.
\end{enumerate}
\end{lemma}

\begin{proof} 
For $(a)$, we apply Theorem~\ref{T:concentration}. First, we estimate the conditional ranges.  From Lemma~\ref{cut_bin}, we get that, for all $k\in[n]$ 
%and $\omega\in\mathcal{E}_k$, $|\mathcal{I}_k|\leq 2\sqrt{X_k'}\log n$,
$$
\mathrm{ran}_k [Y_{k+1}]\leq 2\sqrt{X_k'}\log n\, \one_{\mathcal{E}_k}=2\sqrt{Y_k+k}\log n \,\one_{\mathcal{E}_k}.
$$

We prove by induction on $k$ that ${\sf P}(Y_k> k\log^4 n)\leq\exp[-c\log^2 n]$, where $c>0$ does not depend on $k$ and $n$. For $k=1$, we have $\Prob(Y_1>\log^4 n)\leq \Prob(Y_1>\log n)=0$.

Assume that $\Prob(Y_j> j\log^4 n-j)\leq\exp[-\log^2 n(1+o(1))]$ 
for all $j\leq k$. Then, with a probability at least $1-n\exp[-\log^2 n(1+o(1))]=1-\exp[-\log^2 n(1+o(1))]$, 
$$
\sum_{j=1}^{k+1}(\mathrm{ran}_{j-1}[Y_j])^2\leq 4\log^2 n\sum_{j=0}^{k}(Y_j+j)\leq 2k^2\log^6 n.
$$
Therefore, by Theorem~\ref{T:concentration}, 
\begin{align*}
\Prob\Big(Y_{k+1}>&(k+1)\log^4 n-(k+1)\Big)
 \\
 &\leq 
2\exp\left[-\frac{(k+1)^2}{k^2}\log^2 n(1+o(1))\right]+2\exp\left[-\log^2 n(1+o(1))\right]
\\&=\exp\left[-\log^2 n(1+o(1))\right].
\end{align*}
This proves (a).

For  (b), observe that,  by the definition of $Y_k$, 
$$
\Prob(\exists k \quad Y_k\neq X'_k)=\Prob\left(\bigcup_k\mathcal{B}_k\setminus\mathcal{E}_k\right)\leq\sum_{k=1}^n\Prob\left(X'_k-1\notin\mathcal{I}_k \mid \mathcal{B}_k\right).
$$ 
%Let us prove that, for $n$ large enough and all $k$, ${\sf P}\left(\bigcup_k\mathcal{B}_k\setminus\mathcal{E}_k\right)\leq\exp[-\frac{1}{5}\ln n\ln\ln n]$.
Each term in the sum above is $e^{-\omega(\log n)}$ by \cref{L:Ch}
and  the difinition of $X'_k$ given in \cref{L:X'}. Part (b) follows.
\end{proof}
\begin{lemma}\label{L:Ch}
For $n$ large enough and all positive integers $a\leq N$, a random variable $\xi \sim \Bin(N,a/N)$ satisfies the following:
$$
\Prob(|\xi-a|>\sqrt{a}\log n)\leq\mathrm{exp}\left(-\dfrac{1}{5}\log n\log\log n\right).
$$
\end{lemma}
\begin{proof}
  By the Chernoff bounds,
\begin{align*}
\Prob(\xi\geq a+\sqrt{a}\log n)&\leq\mathrm{exp}\left[\sqrt{a}\log n-(a+\sqrt{a}\log n)\ln\left(1+\frac{\log n}{\sqrt{a}}\right)\right], 
\\
 \Prob(\xi\leq a-\sqrt{a}\log n)&\leq \mathrm{exp}\left[-\frac{1}{2}\log^2 n\right]
\end{align*}
 It is straightforward to check that the stated bound holds for all possible values of $a$. 
\end{proof}

%Denoting $\sqrt{a}=\beta\ln n$, we get 
%\begin{enumerate}
%\item ${\sf P}(\eta_a\geq a+\sqrt{a}\ln n)\leq\mathrm{exp}\left[-\frac{\ln^2 n}{4}\right]$ for $\beta\geq 2$ since $\ln(1+1/\beta)>1/\beta-1/(2\beta^2)$;
%\item ${\sf P}(\eta_a\geq a+\sqrt{a}\ln n)\leq\mathrm{exp}\left[-\ln n\ln\ln n\left(3\ln\frac{3}{2}-1\right)\right]$ for $\frac{\ln\ln n}{\ln n}\leq\beta\leq 2$ since $(\beta+1)\ln(1+\frac{1}{\beta})$ decreases and achieves its minimum $3\ln\frac{3}{2}>1$ when $\beta=2$;
%\item ${\sf P}(\eta_a\geq a+\sqrt{a}\ln n)\leq\mathrm{exp}\left[-\sqrt{a}\ln n(\ln\ln n-\ln\ln\ln n-1)\right]$ for $\beta\leq \frac{\ln\ln n}{\ln n}$ since $\ln(1+1/\beta)\leq(\ln\ln n-\ln\ln\ln n)$.
%\end{enumerate}
%This finishes the proof since, for good sequences $y_1,\ldots,y_k,x_1,\ldots,x_k$, the random variable $Y_{k+1}$ has binomial distribution with the above parameters conditioning on the event $\{Y_1=y_1,\ldots,Y_k=y_k,X_1=x_1,\ldots,X_k=x_k\}$ and by the last property in Lemma~\ref{cut_bin}.\\

%Let us conclude that, with a probability at least $1-\exp[-c\ln n\ln\ln n]$,
%$$
%|B^k_{T_n}(1)|=X_k\leq Y_k=\tilde Y_k\leq k\ln^4 n.
%$$ 
%The statement of Theorem~\ref{balls} follows by the union bound and since $|B^{\leq d} {T_n}(1)|=1+\sum_{k=1}^d|B^k_{T_n}(1)|$. $\Box$

%%%%%%%%%%%%%%%%%%%%%%%%%%%%%%%%%%%%%%%%%%%%%%%%%%%%%%%%%%%%
%%%%%%%%%%%%%%%%%%%%%%%%%%%%%%%%%%%%%%%%%%%%%%%%%%%%%%%%%%%%
%%%%%%%%%%%%%%%%%%%%%%%%%%%%%%%%%%%%%%%%%%%%%%%%%%%%%%%%%%%%

%%%%%%%%%%%%%%%%%%%%%%%
%%%%%%%%%%%%%%%%%%%%%%%
%%%%%%%%%%%%%%%%%%%%%%%
%%%%%%%%%%%%%%%%%%%%%%%
%%%%%%%%%%%%%%%%%%%%%%%

\end{document}